\documentclass[dvipsnames,reqno,11pt]{amsart}
\usepackage{amsmath}
\usepackage[margin=1.2in]{geometry}
\usepackage{tabularx}
\usepackage{amssymb, verbatim}
\usepackage{pstricks,pst-node,pst-plot}
\usepackage{tikz, tikz-cd}
\usepackage[all,cmtip]{xy}
\usepackage{comment}
\usepackage{mathrsfs} 
\usepackage{amsrefs}
\usepackage{color}
\definecolor{ForestGreen}{RGB}{34,139,34}
\usepackage[]{hyperref}
\hypersetup{
    colorlinks=true,
    linkcolor=blue,
    filecolor=magenta,  
    urlcolor=cyan,
    citecolor=ForestGreen,
    pdfpagemode=FullScreen,
    }
\usepackage{caption}
\usepackage{setspace}
\usepackage{etoolbox}

\title[Period domains]{Period domains for gravitational instantons}

\author[T.-J. Lee]{Tsung-Ju Lee}
\email{tjlee@cmsa.fas.harvard.edu }
\address{Center of Mathematical Sciences and Applications, Harvard University, 
20 Garden Street, Cambridge, MA 02138}
  
\author[Y.-S. Lin]{Yu-Shen Lin}
\email{yslin@bu.edu}
\address{Department of Mathematics, Boston University, 111 Cummington Mall, Boston, MA 02215}

\theoremstyle{plain}
\newtheorem{thm}{Theorem}[section]
\newtheorem{prop}[thm]{Proposition}
\newtheorem{lem}[thm]{Lemma}

\theoremstyle{definition}
\newtheorem{defn}[thm]{Definition}

\newtheorem{rk}[thm]{Remark}

\usepackage{enumitem}
\numberwithin{thm}{section}
\numberwithin{equation}{section}
\counterwithin{table}{section}
\setlist[itemize]{leftmargin=3em}

\newcommand{\be}{\begin{equation}}
\newcommand{\bea}{\begin{eqnarray}}
\newcommand{\eea}{\end{eqnarray}}
\newcommand{\ee}{\end{equation}}

\renewcommand{\leq}{\leqslant}
\renewcommand{\geq}{\geqslant}
\renewcommand{\le}{\leqslant}
\renewcommand{\ge}{\geqslant}
\renewcommand{\epsilon}{\varepsilon}
\renewcommand{\phi}{\varphi}

\begin{document}
\begin{abstract}
	Based on the uniformization theorems of gravitation instantons by Chen--Chen \cite{CC1}, Chen--Viaclovsky \cite{CV}, Collins--Jacob--Lin \cite{CJL3}, and Hein--Sun--Viaclovsky--Zhang \cite{HSVZ2}, we prove that the period maps for the 
  \(\mathrm{ALH}^{\ast}\), \(\mathrm{ALG}\), and \(\mathrm{ALG}^{\ast}\)
  gravitational instantons are surjective. In particular, the period domains of these gravitational instantons are exactly their moduli spaces. 
	 \end{abstract}
\maketitle
\tableofcontents

\section{Introduction}

Gravitational instantons, introduced by Hawking \cite{Haw} for his Euclidean quantum gravity theory, are defined as non-compact complete hyperK\"ahler $4$-manifolds with $L^2$ curvature tensors. From the viewpoint of differential geometry, gravitational instantons arise naturally as a bubbling limit of hyperK\"ahler metrics on \(K3\) surfaces \cites{Fos,HSVZ,CVZ}. Therefore, they can be viewed as the building blocks towards the understanding of $2$-dimensional Calabi--Yau metrics. 
The early discovered gravitational instantons are classified by their volume growths $r^4,r^3,r^2,r$. Those with volume growth $r^4$ are called locally asymptotically Euclidean (ALE), those with volume growth $r^3$ are called locally asymptotically flat (ALF) and the rest two are named 
$\mathrm{ALG}$ and $\mathrm{ALH}$ by induction. Later, Hein \cite{Hein} found two new types of gravitational instantons named as $\mathrm{ALG}^{\ast}$ and $\mathrm{ALH}^{\ast}$. The former has volume growth $r^2$, as the $\mathrm{ALG}$ gravitational instantons, but with a different curvature decay rate while the latter has volume growth $r^{4/3}$. Recently, Sun--Zhang \cite{SZ} used the Cheeger--Fukaya--Gromov theory to prove that any non-flat graviational instanton has a unique asymptotic cone and it must belong to one of the above six types. As a summary, there are six types of gravitational instatons in total: 
\(\mathrm{ALE},\mathrm{ALF},\mathrm{ALG}, \mathrm{ALH}, 
\mathrm{ALG}^{\ast}\), and \(\mathrm{ALH}^{\ast}\).

To further classify the gravitational instantons within each type, people are interested in the following questions:
\begin{enumerate}
  \item What are the possible diffeomorphism types of the gravitational instantons within each type?
  \item What are the possible cohomology classes of the hyperK\"ahler triples for a fixed diffeomorphism type of gravitational instantons?
  \item Does the cohomology classes of the hyperK\"ahler triple uniquely determine the gravitational instantons isometrically? 
\end{enumerate}
The set of possible cohomology classes supporting the hyperK\"ahler triples of gravitational instantons within a fixed diffeomorphism type is usually known as the period domain. The second question can be then rephrased as ``how to characterize the period domain of gravitational instantons within a fixed diffeomorphism type?'' The third question is usually known as the Torelli theorem of gravitational instantons. 

Kronheimer first answered all these questions for $\mathrm{ALE}$ gravitational instantons \cites{Kro,Kro2}. In which case, topologically, the underlying geometry always comes from the crepant resolution of the quotient of $\mathbb{C}^2$ by a finite subgroup of $\mathrm{SU}(2)$. Any triple in $\mathrm{H}^2(X,\mathbb{R})$ can be realized as the cohomology classes of the hyperK\"ahler triples when they do not vanish simultaneously on the $(-2)$-classes in $\mathrm{H}_2(X,\mathbb{Z})$.
Moreover, Kronheimer established a Torelli-type theorem for $\mathrm{ALE}$ gravitational instantons. The analogue theorem for $\mathrm{ALF}$ gravitational instantons has been established by Chen--Chen \cite{CC2}.
For the rest of gravitational instantons, the first question is answered by certain ``uniformization theorems'' (see Section \ref{sec: uniform}): for any gravitational instantons of types $\mathrm{ALG},\mathrm{ALH},\mathrm{ALG}^{\ast},\mathrm{ALH}^{\ast}$, up to a suitable hyperK\"ahler rotation they can be compactified to rational elliptic surfaces by filling in a fibre with monodromy of finite order, smooth fibre, an $\mathrm{I}_k^{\ast}$-fibre or an $\mathrm{I}_k$-fibre respectively \cites{Hein, CC1, CJL, CJL3,HSVZ2}. In particular, there are finitely many diffeomorphism types of the gravitational instantons from the classification of singular fibres of rational elliptic surfaces of Perrson \cite{Per}. The Torelli-type theorems for these gravitational instantons are also established: the $\mathrm{ALH}$ case by Chen--Chen \cite{CC3}, the $\mathrm{ALG}$ and $\mathrm{ALG}^{\ast}$ cases by Chen--Viaclovsky--Zhang \cite{CVZ2} and the $\mathrm{ALH}^{\ast}$ case by the second author with Collins and Jacob \cite{CJL3}. While the questions about characterizations of period domains of gravitational instantons remain open, it is observed that not all the cohomology classes can be realized as those of the hyperK\"ahler triples of gravitational instantons - there are some obvious topological constraints: those homology classes with self-intersection $-2$ can be realized as holomorphic curves after a suitable hyperK\"ahler rotation and particularly the corresponding K\"ahler form can not vanish on it\footnote{If a \((-2)\) class vanishes on the hyperk\"{a}hler triple, then, up to a hyperk\"{a}hler rotation, it can be realized a \((-2)\) curve. We can contract it to get an orbifold. In which case, the Calabi--Yau metric should be replaced by the orbifold Calabi--Yau metric.}. Subsequently, Chen--Viaclovsky--Zhang \cite{CVZ2} conjectured that given a diffeomorphism type of $\mathrm{ALG}$ or $\mathrm{ALG}^{\ast}$ gravitational instanton, any cohomology classes of hyperk\"ahler triples on which do not vanish simultaneously can be realized by a gravitational instanton. One can make a similar conjectural statement for the $\mathrm{ALH}^{\ast}$ gravitational instantons. 

The goal of this manuscript is to study these conjectures. 
Let us outline the organization of this manuscript and, in the meanwhile, briefly explain the idea of the proof of the conjecture for the $\mathrm{ALH}^{\ast}$ case since the ideas for the other two cases are pretty much similar. We treat
\(\mathrm{ALH}^{\ast}\) gravitational instantons in \S\ref{sec: uniform}
and \(\mathrm{ALG}\) as well as \(\mathrm{ALG}^{\ast}\) gravitational instantons
in \S\ref{sec:alg}. 
In \S\ref{subsec:weak-del-pezzo-surfaces}, we recall some basics about pairs \((Y,D)\) with \(Y\) a (weak) del Pezzo surface and \(D\in |-K_{Y}|\) smooth, and the fact that for such a pair \((Y,D)\) the complement \(Y\setminus D\) can support $\mathrm{ALH}^{\ast}$ gravitational instantons. In \S\ref{subsec:construction-of-m-log-CY}--\S\ref{subsec:alh*-period}, we construct pairs \((Y,D)\) to realize cohomology classes in \(\mathrm{H}^{2}(X_{\mathfrak{r}},\mathbb{C})\) of a reference \(\mathrm{ALH}^{\ast}\) gravitational instanton \(X_{\mathfrak{r}}\) as the cohomology classes of the $(2,0)$-form \(\Omega\) on \(X=Y\setminus D\). We also show that any cohomology class 
which is positive on every holomorphic curve in $X$ supports a Ricci-flat metric asymptotic to Calabi ansatz and thus gives an $\mathrm{ALH}^{\ast}$ gravitational instanton. In \S\ref{subsec:mono}, we demonstrate how to use monodromy transformations to reduce all the other cases to the previous one. 
Finally we give a complete proof of the surjectivity of the period map in \S\ref{subsec:alh*-surj}. 
In \S\ref{subsec:construction}, we construct \(\mathrm{ALG}\) and \(\mathrm{ALG}^{\ast}\) pairs \((Y,D)\)
to realize cohomology classes in \(\mathrm{H}^{2}(X_{\mathfrak{r}},\mathbb{C})\) of the complement \(X_{\mathfrak{r}}=Y_{\mathfrak{r}}\setminus D_{\mathfrak{r}}\) of a reference \(\mathrm{ALG}\) or \(\mathrm{ALG}^{\ast}\) pair \((Y_{\mathfrak{r}},D_{\mathfrak{r}})\) as the cohomology classes of the $(2,0)$-form on \(X=Y\setminus D\). In \S\ref{subsec:alg-alg*-gra-ins}, we recall some basics of \(\mathrm{ALG}\) and \(\mathrm{ALG}^{\ast}\) gravitational instantons, including the definition of the period maps as well as the uniformization theorem. Finally in \S\ref{subsec:sur-alg-alg*}, we prove the surjectivity of the period maps. To sum up,
\begin{thm}(=Theorem \ref{main1} and Theorem \ref{main2})
	The period maps for $ALH^*/ALG/ALG^*$ gravitational instantons are all surjective. 
\end{thm}

At the moment when this manuscript was about to be finished, Chen--Viaclovsky--Zhang had a different proof for the conjecture in the cases of $\mathrm{ALG}$ in the second version of their preprint \cite{CVZ2}. On the other hand, it is conjectured that certain gauge theory moduli spaces constructed in Biquard--Boalch \cite{BB} and Cherkis--Kapustin \cite{CK} will achieved all possible periods and known as the modularity conjecture. We will refer the readers to Mazzeo--Fridrickson--Swoboda--Weiss for the progress along this line, which would eventually lead to a different proof of the surjectivity of period maps in the cases of $\mathrm{ALG}$ and $\mathrm{ALG}^{\ast}$. 

\quad\\

\noindent {\bf Acknowledgements.} The second author would like to thank R.~Zhang for bring the problem to his attention. The second author would also like to thank G.~Chen, T.~Collins, A.~Jacob, J.~Viaclovsky, and R.~Zhang for some related discussions. The authors are grateful to S.-T.~Yau for his interest and steadily encouragement. We would like to thank CMSA for providing a good environment for discussions. We would like to thank anonymous referees for their careful reading and valuable comments. The first author is supported by the Simons Collaboration on HMS Grant and the AMS--Simons Travel Grant (2020--2023). The second author is supported by Simons collaboration grant \# 635846 and NSF grant DMS \#2204109.

\section{Period domains of \texorpdfstring{$\mathrm{ALH}^{\ast}$}{ALH*} gravitational instantons} 
\label{sec: uniform}

\subsection{Weak del Pezzo surfaces}
\label{subsec:weak-del-pezzo-surfaces}
A rational surface $Y$ is a weak del Pezzo surface if its anti-canonical divisor \(-K_{Y}\) 
is big and nef. From the classification of compact complex surfaces, one has 
\begin{prop}
\label{prop:weak-del-pezzo}
	Weak del Pezzo surfaces are either blow-up of $\mathbf{P}^2$ at generic $b=9-d$ points with $1\leq d\leq 9$ or
	$\mathbf{P}^1\times \mathbf{P}^1$ or the Hirzebruch surface $\mathbb{F}_2$. Here generic configuration means 
	\begin{itemize}
		\item all points are proper (no multiplicity higher than \(2\));
		\item no three points are on a line;
		\item no six points are on a conic;
		\item no cubic passes through the points with one of 
		them being a singular point of that cubic.
	\end{itemize}	
\end{prop}
From the above proposition, any holomorphic curve in a weak del Pezzo surface has self-intersection at least $-2$. 
The self-intersection number $d=(-K_Y)^2$ is the 
\emph{degree} of the weak del Pezzo surface $Y$. Every weak del Pezzo surface 
admits a smooth anti-canonical divisor.
Thus, there are in total $10$ deformation families of pairs consisting 
of a weak del Pezzo surface and a smooth anti-canonical divisor: one deformation family for each $d\neq 8$ and two for $d=8$. 
Notice that the Hirzebruch surface $\mathbb{F}_2$ is in 
the deformation family of $\mathbf{P}^1\times \mathbf{P}^1$. 
For notational simplicity, we shall denote the 
degree of $\mathbf{P}^1\times \mathbf{P}^1$ or $\mathbb{F}_2$ by $d=8'$.

To describe the period domains of $\mathrm{ALH}^{\ast}$ gravitational instantons,
we need to compute $\mathrm{H}^2(X,\mathbb{Z})$ and $\mathrm{H}^2(Y,\mathbb{Z})$. 
We use the long exact sequence of the pair 
$(Y,D)$ (cf.~\cite{Lo}*{\S I.5.1}),
\begin{equation}
\label{eq:long-exact-sequence}
0\to \mathrm{H}^{1}(D,\mathbb{Z})\to
\mathrm{H}_{2}(X,\mathbb{Z})\to \mathrm{H}_{2}(Y,\mathbb{Z})
\to \mathrm{H}^{2}(D,\mathbb{Z})\to \mathrm{H}_{1}(X,\mathbb{Z})\to 0.
\end{equation} Notice that $\mathrm{H}_1(X,\mathbb{Z})$ is torsion and 
in particular \(\operatorname{rank}_{\mathbb{Z}}\mathrm{H}_{2}(X,\mathbb{Z})=11-d\) is determined by
the degree of the weak del Pezzo surface $Y$. 
The connecting homomorphism \(\mathrm{H}_{2}(Y,\mathbb{Z})\to
\mathrm{H}^{2}(D,\mathbb{Z})\) in 
\eqref{eq:long-exact-sequence}
is identified with the signed intersection
\begin{equation*}
\varphi_{[D]}\colon [C]\mapsto ([D]\mapsto [C]\cdot[D])
\end{equation*}
and we obtain a short exact
\begin{equation}
\label{eq:homology-of-x}
0\to \mathrm{H}^{1}(D,\mathbb{Z})\to
\mathrm{H}_{2}(X,\mathbb{Z})\to \mathrm{ker}(\varphi_{[D]})\to 0
\end{equation} 
where \(\varphi_{[D]}\) denotes the signed intersection map.
Via Poincar\'{e} duality, we can further identify \(\mathrm{ker}(\varphi_{[D]})\)
with $[D]^{\perp}$, the subgroup of $\operatorname{Pic}(Y)$ with zero pairing with the Poincar\'{e} dual of $[D]$. 

It is known that the middle cohomology group of a smooth weak del Pezzo
surface is isomorphic to either \(\mathbb{Z}^{1,9-d}\) or \(\mathrm{U}_{2}\)
(the hyperbolic lattice of rank two). 
Let \(Y\) be a weak del Pezzo surface of degree \(d\ne 8'\)
and let \(\pi\colon Y\to\mathbf{P}^{2}\) be a blow-up (at \(b=9-d\) points)
realization of \(Y\).
Denote by \(E_{1},\ldots,E_{b}\) the exceptional divisors of \(\pi\)
and by \(H\) the hyperplane in \(\mathbf{P}^{2}\).
Then the assignments 
\(e_{0}\mapsto [H]\) and \(e_{i}\mapsto [E_{i}]\) 
(the pullbacks are omitted)
give rise to an isomorphism of lattices
\(\mathbb{Z}^{1,b}\to \mathrm{H}^{2}(Y,\mathbb{Z})\).
The anti-canonical divisor of \(Y\) is linearly equivalent to
\(3H-E_{1}-\cdots-E_{b}\).
Moreover,
\begin{equation*}
\{H-3E_{1},E_{i}-E_{i+1},~i=1,\ldots,b-1\}
\end{equation*}
is a basis of \([D]^{\perp}\) with \(D\in |-K_{Y}|\).

If \(Y\) is such that \(d=8'\),
it is straightforward to check that \(\mathrm{H}^{2}(Y,\mathbb{Z})\cong\mathrm{U}_{2}\)
under the basis \(\{[\ell_{1}],[\ell_{2}]\}\)
where \(\ell_{i}\)'s are (parallel transport of) the rulings
in \(\mathbf{P}^{1}\times\mathbf{P}^{1}\). 
The anti-canonical divisor is linearly equivalent to \(2[\ell_{1}]+2[\ell_{2}]\) and
\([D]^{\perp}\cong\langle \ell_{1}-\ell_{2}\rangle\)
for \(D\in |-K_{Y}|\).

\subsection{Constructions of \texorpdfstring{\((Y,D)\)}{(Y,D)} for
\texorpdfstring{\(\mathrm{ALH}^{\ast}\)}{ALH*} gravitational instantons}
\label{subsec:construction-of-m-log-CY}
The purpose of this subsection is to 
construct reference marked log Calabi--Yau pairs
coming from (weak) del Pezzo surfaces.
By a \emph{marked log Calabi--Yau
pair} we mean a log Calabi--Yau pair \((Y,D)\)
together with a basis \(\mathcal{B}\) of \(\mathrm{H}_{2}(X,\mathbb{Z})\)
with \(X:=Y\setminus D\), called the 
\emph{distinguished basis}.
We will treat the cases \(1\le d\le 9\) and \(d=8'\) separately.

We now construct a
marked log Calabi--Yau pair
\((Y_{\mathfrak{r},d},D_{\mathfrak{r},d})\) 
for each \(1\le d\le 9\) (\(d\ne 8'\))
where \(Y_{\mathfrak{r},d}\) is 
a smooth del Pezzo surface of degree \(d\) and 
\(D_{\mathfrak{r},d}\) is a smooth anti-canonical
divisor.

For \(d=9\), we simply take \(Y_{\mathfrak{r},9}=\mathbf{P}^{2}\)
and \(D_{\mathfrak{r},9}\) to be a smooth elliptic curve. 
From the long exact sequence of compactly supported cohomology
\begin{equation}
0\to \mathrm{H}^{1}_{\mathrm{c}}(D_{\mathfrak{r},9},\mathbb{Z})\to 
\mathrm{H}^{2}_{\mathrm{c}}(X_{\mathfrak{r},9},\mathbb{Z})\to
[D_{\mathfrak{r},9}]^{\perp}=\{0\},
\end{equation}
we have an isomorphism
\begin{equation}
\label{eq:delta-s1-bundles}
\mathrm{H}^{1}_{\mathrm{c}}(D_{\mathfrak{r},9},\mathbb{Z})\xrightarrow{\delta}
\mathrm{H}_{\mathrm{c}}^{2}(X_{\mathfrak{r},9},\mathbb{Z})
\end{equation}
which, under Poincar\'{e} duality, is identified with ``taking an \(S^{1}\)-bundle.''
\(\delta\) is also known as the Leray coboundary map.

Choose a symplectic basis \(\{\alpha_{\mathfrak{r}},\beta_{\mathfrak{r}}\}\)
of \(\mathrm{H}_{1}(D_{\mathfrak{r},9},\mathbb{Z})\cong
\mathrm{H}^{1}_{\mathrm{c}}(D_{\mathfrak{r},9},\mathbb{Z})\) and denote their image in 
\(\mathrm{H}_{2}(X_{\mathfrak{r},9},\mathbb{Z})\) by the same notation.
Then \((Y_{\mathfrak{r},9},D_{\mathfrak{r},9})\) and \(\mathcal{B}_{\mathfrak{r},9}
=\{\alpha_{\mathfrak{r}},\beta_{\mathfrak{r}}\}\) 
form our reference marked log Calabi--Yau pair in degree \(9\).

To continue,
we pick \(8\) distinct points \(q_{\mathfrak{r},1},\ldots,q_{\mathfrak{r},8}\in 
D_{\mathfrak{r},9}\). 
For the case \(d= 8\), we take \(Y_{\mathfrak{r},8}=
\mathrm{Bl}_{\{q_{\mathfrak{r},1}\}}\mathbf{P}^{2}\)
and \(D_{\mathfrak{r},8}\) to be the proper transform of \(D_{\mathfrak{r},9}\). 
Notice that \(D_{\mathfrak{r},8}\in |-K_{Y_{\mathfrak{r},8}}|\) 
since \(q_{\mathfrak{r},1}\) belongs to \(D_{\mathfrak{r},9}\).
Put \(X_{\mathfrak{r},8}:=Y_{\mathfrak{r},8}\setminus D_{\mathfrak{r},8}\) as before. Since 
\(D_{\mathfrak{r},8}\cong D_{\mathfrak{r},9}\), 
we can still (and should) use 
\(\{\alpha_{\mathfrak{r}},\beta_{\mathfrak{r}}\}\) as our basis
of \(\mathrm{H}^{1}_{\mathrm{c}}(D_{\mathfrak{r},8},\mathbb{Z})\). 
Denote their image in \(\mathrm{H}_{2}(X_{\mathfrak{r},8},\mathbb{Z})\) by the same notation.
Moreover, \([D_{\mathfrak{r},8}]^{\perp}=[3H-E_{\mathfrak{r},1}]^{\perp}
\cong \langle H-3E_{\mathfrak{r},1}
\rangle_{\mathbb{Z}}\) (\(E_{\mathfrak{r},1}\) is the exceptional divisor
over \(q_{\mathfrak{r},1}\)).
We fix once for all a lifting 
\(\gamma_{\mathfrak{r},1}\in \mathrm{H}_{2}(X_{\mathfrak{r},8},\mathbb{Z})\) 
of \(H-3E_{\mathfrak{r},1}\)
and therefore we achieve a distinguished 
basis \(\mathcal{B}_{\mathfrak{r},8}=
\{\alpha_{\mathfrak{r}},\beta_{\mathfrak{r}},\gamma_{\mathfrak{r},1}\}\) 
of \(\mathrm{H}_{2}(X_{\mathfrak{r},8},\mathbb{Z})\).

We can construct reference marked log Calabi--Yau pairs inductively. 
For the degree \(d\) model \((Y_{\mathfrak{r},d},D_{\mathfrak{r},d})\), 
we blow-up our degree \(d+1\) model \(Y_{\mathfrak{r},d+1}\)
at (the proper transform of) \(q_{\mathfrak{r},b}\)
and we set \(D_{\mathfrak{r},d}\) to be 
the proper transform of \(D_{\mathfrak{r},d+1}\). 
In the present case,
\begin{equation}
[D_{\mathfrak{r},d}]^{\perp} = 
\langle H-3E_{\mathfrak{r},1},
E_{\mathfrak{r},1}-E_{\mathfrak{r},2},
\ldots,E_{\mathfrak{r},b-1}-E_{\mathfrak{r},b}\rangle_{\mathbb{Z}}.~
(\mbox{Recall that \(b=9-d\)}.)
\end{equation}
Here the pullback is omitted.
We may choose the liftings 
\(\gamma_{\mathfrak{r},1},\ldots,\gamma_{\mathfrak{r},b}\in
\mathrm{H}^{2}(X_{\mathfrak{r},d},\mathbb{Z})\) 
in a such way that they are identified with the corresponding 
elements in the distinguished basis \(\mathcal{B}_{\mathfrak{r},d+1}\)
in the degree \(d+1\) model under the blow-up
\(Y_{\mathfrak{r},d}\to Y_{\mathfrak{r},d+1}\).
Then \((Y_{\mathfrak{r},d},D_{\mathfrak{r},d})\) and the basis
\(\{\alpha_{\mathfrak{r}},\beta_{\mathfrak{r}},\gamma_{\mathfrak{r},1},
\ldots,\gamma_{\mathfrak{r},b}\}\)
of \(\mathrm{H}_{2}(X_{\mathfrak{r},d},\mathbb{Z})\)
give our degree \(d\) model.

For \(d=8'\), we begin with \(\mathbf{P}^{2}\) and a 
smooth elliptic curve \(E\subset\mathbf{P}^{2}\). Pick 
\(p_{\mathfrak{r}},q_{\mathfrak{r}}\in E\) 
such that \(L:=\overline{p_{\mathfrak{r}}q_{\mathfrak{r}}}\) intersects \(E\)
transversally and consider the blow-up
\begin{equation*}
\begin{tikzcd}
&\mathrm{Bl}_{\{p_{\mathfrak{r}},q_{\mathfrak{r}}\}}
\mathbf{P}^{2}\ar[r,"\pi"]&\mathbf{P}^{2}.
\end{tikzcd}
\end{equation*}
Denote by \(E_{\mathfrak{r},p}\) and \(E_{\mathfrak{r},q}\) the exceptional divisors
over \(p_{\mathfrak{r}}\) and \(q_{\mathfrak{r}}\). 
The proper transform \(\bar{L}\) of \(L\) becomes a \((-1)\) curve. We
can contract \(\bar{L}\) and obtain a blow-down \(\rho\colon 
\mathrm{Bl}_{\{p_{\mathfrak{r}},q_{\mathfrak{r}}\}}\mathbf{P}^{2}\to Y\) to a smooth
projective surface \(Y_{\mathfrak{r}}\). 
\begin{equation*}
\begin{tikzcd}
&\mathrm{Bl}_{\{p_{\mathfrak{r}},q_{\mathfrak{r}}\}}
\mathbf{P}^{2}\ar[d,"\rho"]\ar[r,"\pi"]&\mathbf{P}^{2}\\
& Y_{\mathfrak{r}}
\end{tikzcd}
\end{equation*}
By surface classification, we have \(Y_{\mathfrak{r}}
\cong\mathbf{P}^{1}\times\mathbf{P}^{1}\)
and \(\rho(E_{\mathfrak{r},p})\) and \(\rho(E_{\mathfrak{r},q})\) are the rulings.
Put \(Y_{\mathfrak{r},8'}=Y_{\mathfrak{r}}\). 
The proper transform \(\bar{E}\) of \(E\) projects down
to a smooth elliptic curve \(D_{\mathfrak{r},8'}\) and
we put \(X_{\mathfrak{r},8'}=
Y_{\mathfrak{r},8'}\setminus D_{\mathfrak{r},8'}\).
Again we fix a symplectic basis \(\{\alpha_{\mathfrak{r}},
\beta_{\mathfrak{r}}\}\) of
\(\mathrm{H}_{1}(D_{\mathfrak{r},8'},\mathbb{Z})\)
and denote their images in 
\(\mathrm{H}_{2}(X_{\mathfrak{r},8'},\mathbb{Z})\) by the same notation.
In this case, \([D_{\mathfrak{r},8'}]^{\perp}\) is 
generated by the difference of the rulings and 
we shall again fix once for all a lifting \(\gamma_{\mathfrak{r}}\) of
\(\rho_{\ast}([E_{q}]-[E_{p}])\) so that
\(\{\alpha_{\mathfrak{r}},\beta_{\mathfrak{r}},\gamma_{\mathfrak{r}}\}\) is our distinguished basis.

\subsection{\texorpdfstring{$\mathrm{ALH}^{\ast}$}{ALH*} gravitational instantons}
\label{subsec:alh*-gravitational-instantons}
$\mathrm{ALH}^{\ast}$ gravitational instantons intuitively are the gravitational instantons which asymptotics to Calabi ansatz. We first explain the construction of Calabi ansatz. Let $D$ be an elliptic curve and $p\colon L\rightarrow D$ be a positive line bundle of degree $d$. Let $Y_{\mathcal{C}}$ be the total space of $L$ with projection $\pi_{\mathcal{C}}\colon Y_{\mathcal{C}}\rightarrow D$. Let $X_{\mathcal{C}}$ be the complement of the zero section in $Y_{\mathcal{C}}$. Let $h$ be the unique hermitian metric on $L$ whose curvature form is $\omega_D$ with the normalization $\int_D \omega_D=2\pi d$. If $z$ is the coordinate on $D$ and $\xi$ is a local trivialization of $L$, we get coordinates on $L$ via $(z,w)\mapsto(z,w\xi)$. The Calabi ansatz is then given by 
\begin{align*}
\omega_{\mathcal{C}}=\frac{2ic}{3}\partial \bar{\partial}\left( -\log{|\xi|^2_h} \right)^{\frac{3}{2}},\qquad \Omega_{\mathcal{C}}=c\pi_{\mathcal{C}}^*\Omega_D\wedge \frac{\mathrm{d}w}{w},
\end{align*} where $c$ is any positive real number and $\Omega_D$ is a holomorphic volume form such that
\begin{align*}
\frac{i}{2}\int_{D} \frac{\Omega_D}{2\pi i}\wedge \overline{\bigg(\frac{\Omega_D}{2\pi i}\bigg)}=2\pi d.
\end{align*}
It is straightforward to check that $(\omega_{\mathcal{C}},\Omega_{\mathcal{C}})$ is a hyperK\"ahler triple, i.e., $2\omega_{\mathcal{C}}^2=\Omega_{\mathcal{C}}\wedge \bar{\Omega}_{\mathcal{C}}$. 

\begin{defn}\label{def: ALH*} Given $d\in \mathbb{N},\tau\in \mathfrak{h}/\mbox{SL}(2,\mathbb{Z}), c>0$.
 An $\mathrm{ALH}^{\ast}$ gravitational instanton (of type $(d,\tau,c)$) is a triple $(X,\omega,\Omega)$, 
 where $X$ is a non-compact complete hyperK\"ahler $4$-manifold with a K\"ahler form $\omega$, 
 and a holomorphic volume form $\Omega$ such that 
  \begin{enumerate}
  	\item $2\omega^2=\Omega\wedge \bar{\Omega}$ and 
  	\item there exists a compact set $K\subseteq X$, an $\epsilon>0$ and a diffeomorphism 
	$F\colon X_{\mathcal{C}}\cong X\setminus K$ such that 
  	 \begin{align*}
  	  \|F^{\ast}\omega-\omega_{\mathcal{C}}\|_{g_{\mathcal{C}}}=O(r^{-k-\epsilon}), 
	  ~\quad~\|F^{\ast}\Omega-\Omega_{\mathcal{C}}\|_{g_{\mathcal{C}}}=O(r^{-k-\epsilon}), 
  	 \end{align*} where $r$ is the distance to a fixed point in $X_{\mathcal{C}}$. 
  \end{enumerate}
\end{defn}	
\begin{rk}
\label{rk:def-notation}
	\begin{enumerate}
		\item
   From \eqref{eq:long-exact-sequence}, 
   let $\alpha,\beta\in \operatorname{Im}(\mathrm{H}^1(D,\mathbb{Z})\to \mathrm{H}_{2}(X,\mathbb{Z}))$ 
   be the image of an oriented basis of $\mathrm{H}^1(D,\mathbb{Z})$. Then
    \begin{align*}
     \{\Omega\}:=\frac{\int_{\beta}\Omega}{\int_{\alpha}\Omega}=\tau\mod{\mathrm{SL}(2,\mathbb{Z})}
     \end{align*} is an invariant of the $\mathrm{ALH}^{\ast}$ gravitational instanton. 
   \item Any $\mathrm{ALH}^{\ast}$ gravitational instanton can be compactified to a rational elliptic surface by adding an $\mathrm{I}_d$-fibre at infinity \cites{CJL3,HSVZ2}. From the classification of singular fibres of rational elliptic surfaces \cite{Per}, one has $1\leq d\leq 9$. We will use $\mathrm{ALH}^{\ast}_{d}$ gravitational instanton to indicate the diffeomorphism type of an $\mathrm{ALH}^{\ast}$ gravitational instanton, with $1\leq d\leq 9$ or $d=8'$. See the discussion after \cite{CJL2}*{Proposition 5.4}.
 \end{enumerate}   
\end{rk}

It is proven that any $\mathrm{ALH}^*$ gravitational instanton
can be compactified to 
a weak del Pezzo surface \cites{HSVZ2, CL} by adding a smooth anti-canonical divisor at infinity with modulus $\tau$. 
It is natural to introduce the markings for $\mathrm{ALH}^{\ast}$ gravitational instantons. 

From now on, to ease the notation,
we will omit \(d\) in the subscript most of the time
and only specify it when it plays a role in our discussion.

\begin{defn}\label{def: marked ALH*} Fix a reference $\mathrm{ALH}^{\ast}$ 
gravitational instanton of type $(d,\tau,c)$ with
$d\in\{1,\ldots,9\}$, $\tau\in \mathfrak{h}$, and $c>0$ and an ambient space $X_{\mathfrak{r}}$.
	A quadruple $(X,\omega,\Omega,\mu)$ is called 
  a \emph{marked $\mathrm{ALH}^{\ast}$ gravitational instanton} of type $(d,\tau,c)$ if
	it satisfies
	\begin{enumerate}
	\item $(X,\omega,\Omega)$ is an $\mathrm{ALH}^{\ast}$ gravitational instanton of type $(d,\tau,c)$;
	\item $\mu\colon X_{\mathfrak{r}}\rightarrow X$ is a diffeomorphism
	from the complement \(X_{\mathfrak{r}}:=Y_{\mathfrak{r}}\setminus D_{\mathfrak{r}}\) of
	our marked log Calabi--Yau pair \((Y_{\mathfrak{r}},D_{\mathfrak{r}})\). 
	\end{enumerate}
	Two marked $\mathrm{ALH}^{\ast}$ gravitational instantons $(X_i,\omega_i,\Omega_i,\mu_i)$ are isomorphic if there exists a diffeomorphism $f\colon X_2\rightarrow X_1$ such that $f^*\omega_1=\omega_2$, $f^*\Omega_1=\Omega_2$ and $\mu_1^{\ast}=\mu_2^{\ast}\circ f^*$. Denote $\mbox{mALH}^*(d,\tau,c)$ be the set of marked $\mathrm{ALH}^{\ast}$ gravitational instantons of type $(d,\tau,c)$. 
\end{defn}
Now we fixed a reference $\mathrm{ALH}^{\ast}$ gravitational instanton 
$(X_{\mathfrak{r}},\omega_{\mathfrak{r}},\Omega_{\mathfrak{r}})$ for \((d,\tau,c)\) as above.
We define the period domain of $\mathrm{ALH}^{\ast}$ gravitational instanton $\mathcal{P}\Omega(d,\tau,c)$ to be the
subset of $\mathrm{H}^2(X_{\mathfrak{r}},\mathbb{R})\times \mathrm{H}^2(X_{\mathfrak{r}},\mathbb{C})$ 
consisting of pairs $([\omega],[\Omega])$ such that 
\begin{enumerate}
	\item if $[C]\in \mathrm{H}_2(X_{\mathfrak{r}},\mathbb{Z})$ and
	\([C]^{2}=-2\), then $|[\omega]\cdot [C]|^2+|[\Omega]\cdot [C]|^2\neq 0$. 
	\item $[\omega]$ vanishes on $\operatorname{Im}(\mathrm{H}^1(D_{\mathfrak{r}},
	\mathbb{Z})\rightarrow \mathrm{H}_2(X_{\mathfrak{r}},\mathbb{Z}))$. 
	\item $\{\Omega\}=\tau\mod{\mathrm{SL}(2,\mathbb{Z})}$.
\end{enumerate} 
The period map for $\mathrm{ALH}^{\ast}$ gravitational instantons is then defined to be
\begin{align*}
   \mathcal{P}(d,\tau,c)\colon \mbox{mALH}^*(d,\tau,c) &\rightarrow \mathcal{P}\Omega(d,\tau,c)\\ 
   (X,\omega,\Omega,\mu)&\mapsto (\mu^*[\omega],\mu^*[\Omega]) .
\end{align*}
The goal of this section is to prove the following theorem.
\begin{thm} \label{main1}
	For each $(d,\tau,c)$
	with \(d\in\{1,\ldots,9\}\), \(\tau\in\mathfrak{h}\),
	and \(c>0\) as above, the period map $\mathcal{P}(d,\tau,c)$ is surjective.  
\end{thm}

\subsection{Period domains for holomorphic \texorpdfstring{\(2\)}{2}-forms}
\label{subsec:alh*-period}
Adopting the construction of references log Calabi--Yau pairs in
\S\ref{subsec:weak-del-pezzo-surfaces}, 
we can achieve the following theorem regarding the surjectivity of the period map.
\begin{thm}[Surjectivity of the periods of the $(2,0)$-forms]
	\label{thm:surjectivity-of-periods}
	Given complex numbers \(d_{1}\), \(d_{2}\) satisfying
	\(d_{1}\slash d_{2}\in\mathfrak{h}\) (in particular, \(d_{1}\) and \(d_{2}\) are
  non-zero)
	and \(c_{i}\in\mathbb{C}\), \(1\le i\le b=9-d\), 
	let 
	\begin{equation}
  \label{eq:coh-criterion-alh}
	[\Omega']=d_{1}\mathrm{PD}(\alpha_{\mathfrak{r}})+d_{2}\mathrm{PD}(\beta_{\mathfrak{r}})
	+\sum_{i=1}^{b} 
	c_{i} \mathrm{PD}(\gamma_{\mathfrak{r},i})\in\mathrm{H}^{2}(X_{\mathfrak{r}},\mathbb{C}).
	\end{equation} 
	There exists a marked log Calabi--Yau pair \((Y,D)\)
	and a diffeomorphism \(\mu\colon X_{\mathfrak{r}}\to X\) 
  with \(X=Y\setminus D\) such that
	\begin{equation}
	\mu^{\ast}[\Omega] = [\Omega'],
	\end{equation}
  where \(\Omega\) is a holomorphic \(2\)-form on \(X\),
  that is, any cohomology class in \(\mathrm{H}^{2}(X_{\mathfrak{r}},\mathbb{C})\)
	satisfying the condition \eqref{eq:coh-criterion-alh}
	can be realized as a cohomology class of 
  a holomorphic \(2\)-form on
	some log Calabi--Yau pair. 
\end{thm}

\begin{proof}
  We will construct a marked log Calabi--Yau pair \((Y,D)\)
  such that the complement \(X:=Y\setminus D\) supports 
  a holomorphic \(2\)-form realizing the class \([\Omega']\). 
  
    

  \noindent (a) We deal with the case \(d=9\). In
  this case, we have \(b=0\)
  and 
  \begin{equation*}
  [\Omega'] = d_{1}\mathrm{PD}(\alpha_{\mathfrak{r}})+
  d_{2}\mathrm{PD}(\beta_{\mathfrak{r}}).
  \end{equation*}
  We will construct \(X\) as a complement of an elliptic curve
  in \(\mathbf{P}^{2}\).
  
  
  Put \(\tau = d_{1}\slash d_{2}\in\mathfrak{h}\) and
  let \(X:=\mathbf{P}^{2}\setminus D\) where
  \(D\) is an elliptic curve with modulus \(\tau\);
  \begin{equation*}
  \mathbb{C}\slash \Lambda_{\tau} \cong D \subset \mathbf{P}^{2},~
  ~\Lambda_{\tau}=\mathbb{Z}\oplus\mathbb{Z}\tau.
  \end{equation*} 
  Let \(\Omega\) be a meromorphic \(2\)-form on \(\mathbf{P}^{2}\) with
  a simple pole along \(D\). Notice
  that \(\Omega\) is unique up to a constant.
  By the residue formula, we have
  \begin{equation*}
  \int_{\delta(\alpha)}\Omega=
  \int_{\alpha} \operatorname{Res}\Omega,
  ~\mbox{and}~
  \int_{\delta(\beta)}\Omega=\int_{\beta} \operatorname{Res}
  \Omega
  \end{equation*}
  where \(\{\alpha,\beta\}\)
  is a symplectic basis of \(\mathrm{H}_{1}(D,\mathbb{Z})\)
  and \(\delta\) is the connecting homomorphism
  in \eqref{eq:delta-s1-bundles}.
  Rescaling \(\Omega\) if necessary, we may assume
  \begin{equation*}
  \int_{\alpha}
  \operatorname{Res}\Omega=
  \int_{\delta(\alpha)} \Omega =1.
  \end{equation*}
  Then
  \begin{equation}
  \label{eq:tau-equation}
  \int_{\beta}
  \operatorname{Res}\Omega=
  \int_{\delta(\beta)} \Omega\equiv\tau 
  \mod{\mathrm{SL}(2,\mathbb{Z})}.
  \end{equation}
  We can lift the congruence in \eqref{eq:tau-equation}
  to an equality in \(\mathfrak{h}\).
  Indeed, we can find a path \(\Gamma\) in 
  \(\mathrm{H}^{0}(\mathbf{P}^{2},\mathcal{O}(3))_{\mathrm{sm}}\)
  (the space of smooth sections) such that
  \(\alpha\) (resp.~\(\beta\)) 
  is the parallel transport of \(\alpha_{\mathfrak{r}}\) 
  (resp.~\(\beta_{\mathfrak{r}}\))
  along \(\Gamma\) since 
  the monodromies
  for the family of elliptic curves in \(\mathbf{P}^{2}\)
  generate \(\mathrm{SL}(2,\mathbb{Z})\).
  Consequently, \(\Gamma\) gives rise to 
  a marking \(\mu\colon X_{\mathfrak{r}}\to X\)
  satisfying 
  \begin{equation*}
  \mu^{\ast}(\mathrm{PD}(\delta(\alpha))=\mathrm{PD}(\delta(\alpha_{\mathfrak{r}}))~\mbox{and}~
  \mu^{\ast}(\mathrm{PD}(\delta(\beta)))=\mathrm{PD}(\delta(\beta_{\mathfrak{r}})).
  \end{equation*}
  To ease the notation, we will drop \(\delta(-)\) 
  and simply write \(\alpha\in \mathrm{H}_{2}(X,\mathbb{Z})\) 
  instead of \(\delta(\alpha)\).
  Adapting our convention, the equalities are transformed into
  \begin{equation*}
  \mu^{\ast}(\mathrm{PD}(\alpha))=\mathrm{PD}(\alpha_{\mathfrak{r}})~\mbox{and}~
  \mu^{\ast}(\mathrm{PD}(\beta))=\mathrm{PD}(\beta_{\mathfrak{r}})
  \end{equation*}
  when the context is clear. 
  Then the marked log Calabi--Yau pair \((X,D)\) together with
  the basis \(\{\alpha,\beta\}\subset\mathrm{H}_{2}(X,\mathbb{Z})\)
  and \(\mu\) is what we want. Indeed, because
  \begin{equation*}
  d_{2}=d_{2}\int_{\alpha}\Omega = 
  d_{2}\int_{X}\Omega\wedge\mathrm{PD}(\alpha)~
  \mbox{and}~
  d_{1}=d_{2}\int_{\beta}\Omega = 
  d_{2}\int_{X}\Omega\wedge\mathrm{PD}(\beta),
  \end{equation*}
  we have
  \begin{equation}
  \mu^{\ast}[d_{2}\Omega] = d_{1}\mathrm{PD}(\alpha_{\mathfrak{r}})+
  d_{2}\mathrm{PD}(\beta_{\mathfrak{r}})=[\Omega'].
  \end{equation}
  \noindent (b) We now deal with the case \(d=8\). Let 
  \begin{equation}
  \label{eq:degree-8-hol-period}
  [\Omega']=d_{1}\mathrm{PD}(\alpha_{\mathfrak{r}})+d_{2}\mathrm{PD}(\beta_{\mathfrak{r}})+
  c_{1}\mathrm{PD}(\gamma_{\mathfrak{m},1})
  \in\mathrm{H}^{2}(X_{\mathfrak{r}},\mathbb{C})
  \end{equation}
  with \(d_{1}\slash d_{2}\in\mathfrak{h}\)
  and \(c_{1}\in\mathbb{C}\). 
  By the argument in (a),
  we can find a smooth cubic \(E\subset\mathbf{P}^{2}\), a symplectic basis
  \(\{\alpha,\beta\}\subset\mathrm{H}_{2}(V,\mathbb{Z})\) with \(V=\mathbf{P}^{2}\setminus E\),
  and a diffeomorphism \(\nu\colon X_{\mathfrak{r},9}\to V\) along
  a curve \(\Gamma\) such that
  \begin{equation}
  \nu^{\ast}[\Omega]
  =d_{1}\mathrm{PD}(\alpha_{\mathfrak{r}})+d_{2}\mathrm{PD}(\beta_{\mathfrak{r}}).
  \end{equation}
  Here \(\Omega\) is a holomorphic \(2\)-form
  on \(V\) with the normalization 
  \begin{equation}
  \int_{\alpha} \operatorname{Res}\Omega = d_{2},
  \end{equation}
  and the cycles \(\alpha_{\mathfrak{r}}\) and \(\beta_{\mathfrak{r}}\)
  in \eqref{eq:degree-8-hol-period}
  are regarded as cycles in \(\mathrm{H}_{2}(X_{\mathfrak{r},9},\mathbb{Z})\).

  Let \(p_{1}\in E\) be the parallel transport of \(q_{\mathfrak{r},1}\in 
  D_{\mathfrak{r},9}\) 
  along \(\Gamma\) and consider the blow-up \(\pi'\colon 
  Y'=\mathrm{Bl}_{p_{1}}\mathbf{P}^{2}\to\mathbf{P}^{2}\).
  We have 
  \begin{equation*}
  (\pi')^{\ast}\Omega_{\mathbf{P}^{2}}(E)\cong \Omega_{Y'}(D')
  \end{equation*}
  where \(D'\) is the proper transform of \(E\).
  The path \(\Gamma\)
  determines a marking 
  \(\nu\colon X_{\mathfrak{r},8}\to X'\) for 
  \(X':=Y'\setminus D'\)
  such that \(\nu^{\ast}(\mathrm{PD}(\alpha))=\mathrm{PD}(\alpha_{\mathfrak{r}})\) and
  \(\nu^{\ast}(\mathrm{PD}(\beta))=\mathrm{PD}(\beta_{\mathfrak{r}})\). Now
  let \(\gamma_{1}\in\mathrm{H}_{2}(X',\mathbb{Z})\) such that
  \(\mathrm{PD}(\gamma_{1}')=(\nu^{\ast})^{-1}(\mathrm{PD}(\gamma_{\mathfrak{r},1}))\).
  We get by integration
  \begin{equation}
  \nu^{\ast}[\Omega]=d_{1}\mathrm{PD}(\alpha_{\mathfrak{r}})+d_{2}
  \mathrm{PD}(\beta_{\mathfrak{r}})+c_{1}'\mathrm{PD}(\gamma_{\mathfrak{r},1}),~
  c_{1}'=\int_{\gamma_{1}'}(\pi')^{\ast}\Omega \in\mathbb{C}.
  \end{equation}
  Let \(p\in E\) and \(\Gamma'\) be a smooth curve joining \(p\) and \(p_{1}\).
  Then \(\Gamma'\) gives rise to a diffeomorphism
  \(\rho\colon X'\to X'\) which takes \(p_{1}\) to \(p\).
  Consider the blow-up \(\pi\colon Y=\mathrm{Bl}_{p}\mathbf{P}^{2}\to\mathbf{P}^{2}\).
  The curve \(\Gamma'\) also gives rise to a diffeomorphism 
  \(\rho\colon X'\to X=Y\setminus D\) where
  \(D\) is the proper transform of \(E\) under \(\pi\).
  Let \(\gamma_{1}=\rho_{\ast}(\gamma_{1}')\).
  We can achieve the coefficient \(c_{1}\)
  in \eqref{eq:degree-8-hol-period} by moving \(p\) around in \(E\).
  Indeed, by Lemma \ref{lem:period-calculation} below, we have
  \begin{equation*}
  \int_{{\gamma}_{1}}\Omega \equiv -3\int_{O}^{p} \operatorname{Res}\Omega
  \mod{d_{2}\Lambda_{\tau}}
  \end{equation*}
  where \(O\) is a flex point on \(E\cong D'\) served
  as the additive identity element
  and \(\Lambda_{\tau} = \mathbb{Z}\oplus\mathbb{Z}\tau\) with
  \(\tau = d_{1}\slash d_{2}\).
  Now we choose \(p\in E\) such that
  \begin{equation*}
  c_{1}\equiv -3\int_{O}^{p} \operatorname{Res}\Omega
  \mod{d_{2}\Lambda_{\tau}}.
  \end{equation*}
  We can lift the congruence to an equality
  by adding a loop in \(E\) passing through \(p\) and deforming \(\rho\) accordingly.
  Then the pair \((Y,D)\), the holomorphic top form
  \(\pi^{\ast}\Omega\) and the diffeomorphism \(\mu=\rho\circ\nu\colon X_{\mathfrak{r},8}\to X\)
  are what we are looking for, i.e.,
  \begin{equation*}
  \mu^{\ast}[d_{2}\pi^{\ast}\Omega]=d_{1}\mathrm{PD}(\alpha_{\mathfrak{r}})+
  d_{2}\mathrm{PD}(\beta_{\mathfrak{r}})+
  c_{1}\mathrm{PD}(\gamma_{\mathfrak{m},1})=[\Omega'].
  \end{equation*}
  This proves the theorem when \(d=8\)

  \noindent (c) Now let us deal with the case \(d=7\). 
  Let 
  \begin{equation}
  \label{eq:d=7}
  [\Omega']=d_{1}\mathrm{PD}(\alpha_{\mathfrak{r}})+d_{2}\mathrm{PD}(\beta_{\mathfrak{r}})+
  c_{1}\mathrm{PD}(\gamma_{\mathfrak{r},1})+c_{2}\mathrm{PD}(\gamma_{\mathfrak{r},2})
  \in\mathrm{H}^{2}(X_{\mathfrak{r},7},\mathbb{C})
  \end{equation}
  with \(d_{1}\slash d_{2}\in\mathfrak{h}\)
  and \(c_{1},c_{2}\in\mathbb{C}\). 
  By our discussion in (b), we can find a marked 
  log Calabi--Yau pair \((Y',D')\), 
  a symplectic basis \(\{\alpha,\beta\}\) of \(\mathrm{H}_{1}(D',\mathbb{Z})\),
  and a diffeomorphism \(\nu\colon X_{\mathfrak{r},8}\to X'\)
  with
  \begin{equation}
  \nu^{\ast}[\Omega]=d_{1}\mathrm{PD}(\alpha_{\mathfrak{r}})+
  d_{2}\mathrm{PD}(\beta_{\mathfrak{r}})+
  c_{1}\mathrm{PD}(\gamma_{\mathfrak{r},1}).
  \end{equation}
  Here \(X'=Y'\setminus D'\). We remark that \((Y',D')\)
  is constructed from a blow-up of \(\mathbf{P}^{2}\) and 
  \(D'\) is a proper transform of an elliptic curve in \(\mathbf{P}^{2}\).
  Let \(p_{2}\in D'\) be the image of \(q_{\mathfrak{r},2}\in D_{\mathfrak{r},8}\) (see
  construction in \S\ref{subsec:construction-of-m-log-CY}) under \(\nu\).
  Consider the blow-up \(
  \pi_{2}\colon \mathrm{Bl}_{p_{2}}Y'=:\bar{Y}'\to Y'\) and denote by \(\bar{D}'\)
  the proper transform of \(D'\).
  Let \(E_{2}\) (resp.~\(E_{1}\)) 
  be the exceptional divisor of \(\bar{Y}'\to Y'\) (resp.~
  the pullback of the exceptional divisor of \(Y'\to\mathbf{P}^{2}\)).
  We obtain a diffeomorphism 
  \(\mu\colon X_{\mathfrak{r},7}\to \bar{X}'=\bar{Y}'\setminus \bar{D}'\).
  Define a homology class \(\bar{\gamma}'_{2}\in\mathrm{H}_{2}(\bar{X}',\mathbb{Z})\) via
  \begin{equation}
  \mu^{\ast}(\mathrm{PD}(\bar{\gamma}'_{2})) = \mathrm{PD}(\gamma_{\mathfrak{r},2}).
  \end{equation}
  Let \(p\in D'\) and \(\Gamma'\) be a curve in \(D'\) connecting \(p\)
  and \(p_{2}\). Similar to the case (b), the curve
  \(\Gamma'\) gives rise to a diffeomorphism \(\rho\colon \bar{X}'\to X=\mathrm{Bl}_{p}Y'
  \setminus D\) where \(D\) is the proper transform of \(D'\) and
  put \(\gamma_{2}=\rho_{\ast}(\bar{\gamma}_{2}')\). 
  Then it follows that
  \begin{equation*}
  \int_{\gamma_{2}}\pi^{\ast}\Omega \equiv \int_{p_{1}}^{p}
  \operatorname{Res}\Omega 
  \mod{d_{2}\Lambda_{\tau}}.
  \end{equation*}
  Here we recall that \(p_{1}=\nu(q_{\mathfrak{r},1})\).
  We can achieve \(c_{2}\) in \eqref{eq:d=7} by moving \(p\) around, i.e.,
  we can find an appropriate curve \(\Gamma'\) in \(D\) connecting \(p_{1}\)
  and \(p\) such that
  \begin{equation*}
  c_{2}=\int_{\Gamma'}\operatorname{Res}\Omega = \int_{\gamma_{2}}\Omega.
  \end{equation*}
  This completes the proof when \(d=7\).
  The remaining cases \(1\le d\le 6\) can be done by the same
  procedure inductively.

  \noindent (d) Let us deal with the last case \(d=8'\). Let 
  \begin{equation*}
  [\Omega']=d_{1}\mathrm{PD}(\alpha_{\mathfrak{r}})+d_{2}\mathrm{PD}(\beta_{\mathfrak{r}})+
  c\mathrm{PD}(\gamma_{\mathfrak{r},1})
  \in\mathrm{H}^{2}(X_{\mathfrak{r},8'},\mathbb{C})
  \end{equation*}
  with 
  \(\tau:=d_{1}\slash d_{2}\in\mathfrak{h}\)
  and \(c\in\mathbb{C}\). 
  Similar to the case \(d=9\), let 
  \(E\subset\mathbf{P}^{2}\) be an elliptic curve
  with modulus \(\tau\). 
  Denote by \(\{\alpha,\beta\}\) a sympletic basis of \(\mathrm{H}_{1}(E,\mathbb{Z})\).
  Let \(\Omega\) be a meromorphic 
  \(2\)-form on \(\mathbf{P}^{2}\) having a simple pole along \(E\)
  with the normalization
  \begin{equation*}
  \int_{\alpha} \operatorname{Res}\Omega = 1.
  \end{equation*}
  Then we have
  \begin{equation*}
  \int_{\beta} \operatorname{Res}\Omega\equiv \tau \mod{\mathrm{SL}(2,\mathbb{Z})}.
  \end{equation*}
  Now we can pick \(p,q\in E\) and a curve 
  \(\Gamma\) connecting them such that
  \begin{equation}
  \int_{\Gamma} \operatorname{Res}\Omega = c/d_{2}.
  \end{equation}
  Consider the blow-up \(\pi\colon\mathrm{Bl}_{\{p,q\}}\mathbf{P}^{2}\to\mathbf{P}^{2}\).
  Let \(L\) be the line passing through \(p\) and \(q\).
  Denote by \(E_{p}\) (resp.~\(E_{q}\)) the exceptional divisor over \(p\) (resp.~\(q\))
  and by \(\bar{L}\) (resp.~\(\bar{E}\)) the proper transform of \(L\) (resp.~\(E\)).
  If it happens \(p=q\) (i.e., \(c\equiv 0\mod{\Lambda_{\tau}}\)),
  we shall take \(L\) to be the tangent of \(E\) at \(p\) and
  consider the blow-up at infinitely near points \(p\) and the
  intersection of the proper transform of \(L\) and \(E\).
  In any case, we have \(\pi^{\ast}L = \bar{L}+E_{p}+E_{q}\) and \(\bar{L}\)
  becomes a \((-1)\) curve.
  Let \(\rho\colon\mathrm{Bl}_{\{p,q\}}\mathbf{P}^{2}\to Y\)
  be the blow-down of \(\bar{L}\). 
  When \(p\ne q\), we have \(Y\cong\mathbf{P}^{1}\times\mathbf{P}^{1}\)
  (\(\ell_{1}=\rho(E_{p})\) and \(\ell_{2}=\rho(E_{q})\) give
  the rulings). When \(p=q\), we have \(Y\cong\mathbb{F}_{2}\)
  and \(\rho_{\ast}([E_{q}]-[E_{p}])\) represents the homology 
  class of the unique \((-2)\) curve.
  Put \(D=\rho(\bar{E})\).
  Then \(E\cong \bar{E}\cong D\).
  In any case, \(\gamma:=\rho_{\ast}([E_{q}]-[E_{p}])\) gives
  a homology class in the complement \(X:=Y\setminus D\).
  We can prove that 
  \begin{equation}
  \int_{\gamma}\tilde{\Omega} = \frac{c}{d_{2}} 
  \mod{\Lambda_{\tau}}.
  \end{equation}
  Here \(\tilde{\Omega}\) is the unique meromorphic \(2\)-form on \(Y\)
  having a simple pole along \(D\) such that
  \(\rho^{\ast}\tilde{\Omega} = \pi^{\ast}\Omega\).
  As in the case (a), by choosing a path in 
  \(\mathrm{H}^{0}(\mathbf{P}^{2},\mathcal{O}(3))\)
  appropriately and deforming \(p,q\) on \(E\) suitably, we can find
  a diffeomorphism \(\mu\colon X_{\mathfrak{r},8'}\to X\)
  such that
  \begin{equation*}
  \mu^{\ast}(\mathrm{PD}(\alpha))=\mathrm{PD}(\alpha_{\mathfrak{r}}),~
  \mu^{\ast}(\mathrm{PD}(\beta))=\mathrm{PD}(\beta_{\mathfrak{r}}),~\mbox{and}~
  \mu^{\ast}(\mathrm{PD}(\gamma))=\mathrm{PD}(\gamma_{\mathfrak{r}}).
  \end{equation*}
  Then \(\Omega'=d_{2}\tilde{\Omega}\) is what we need.
\end{proof}

\begin{lem}
	\label{lem:period-calculation}
	Adapt the notation in the proof of 
	Theorem \ref{thm:surjectivity-of-periods}.
	We have
	\begin{equation*}
	\int_{{\gamma}_{1}}\Omega \equiv -3\int_{O}^{p} \operatorname{Res}\Omega~
	\mod{d_{2}\Lambda_{\tau}}
	\end{equation*}
	where \(O\) is a flex point served as
	the additive identity element on \(\bar{D}\)
	and \(\Lambda_{\tau} = \mathbb{Z}\oplus\mathbb{Z}\tau\) with
	\(\tau = d_{2}\slash d_{1}\). 
	From the expression, it is independent of
	the choice of the flex point.
\end{lem} 
\begin{proof}
	Choose a hyperplane \(H\) in \(\mathbf{P}^{2}\)
	passing through \(p\) and intersecting \(D\) at three distinct points,
	say \(D\cap H=\{p,s,t\}\), and
	transversally at \(p\).
	Recall that \([\bar{D}]^{\perp}=\langle H-3E_{1}\rangle\),
	where \(E_{1}\) is the exceptional divisor of 
	\(\pi_{1}\). 
	
	Choose a smooth curve \(\sigma_{1}\) (resp.~\(\sigma_{2}\))
	from \(p\) and \(s\) (resp.~\(p\) and \(t\)).
	We may assume that the relative interior
	of \(\sigma_{i}\) are disjoint.
	By the construction in \cite{Fr}, we can lift the cycle \([H-3E_{1}]\)
	to a cycle \(\delta\) in \(X\) by gluing \(S^{1}\)-bundles over
	\(\sigma_{1}\) and \(\sigma_{2}\).
	Again the lifting is not unique; any two
	liftings differ by an element in \(\mathrm{H}^{1}(D,\mathbb{Z})\).
	Therefore, 
	\begin{equation*}
	\int_{\gamma_{1}}\Omega\equiv \int_{\delta}\Omega~
	\mod{d_{2}\Lambda_{\tau}}.
	\end{equation*}
	Now,
	\begin{align*}
	\int_{\delta}\Omega&=\int_{\sigma_{1}}\operatorname{Res}\Omega+
	\int_{\sigma_{2}}\operatorname{Res}\Omega\\
	&=\int_{p}^{t}\operatorname{Res}\Omega+
	\int_{p}^{s}\operatorname{Res}\Omega\\
	&=\int_{O}^{p}\operatorname{Res}\Omega+
	\int_{O}^{t}\operatorname{Res}\Omega+
	\int_{O}^{s}\operatorname{Res}\Omega
	-3\int_{O}^{p}\operatorname{Res}\Omega\\
	&\equiv -3\int_{O}^{p}\operatorname{Res}\Omega~
	\mod{d_{2}\Lambda_{\tau}}.
	\end{align*}
	The last equation holds since \(p\), \(t\),
	and \(s\) are collinear.
\end{proof}
\subsection{Monodromy of the moduli of pairs}
\label{subsec:mono}
For a smooth projective surface \(S\), 
denote by \(\mathrm{Hilb}^{b}(S)\)
the Hilbert scheme of length \(b\) subscheme on \(S\);
it is a smooth algebraic variety of dimension \(2b\) equipped with
a universal family \(\mathcal{U}\to \mathrm{Hilb}^{b}(S)\).
There exists also a birational morphism (a.k.a.~Hilbert--Chow morphism)
\begin{equation*}
\mathrm{Hilb}^{b}(S) \to \operatorname{Sym}^{b}(S),\hspace{0.5cm}
p\mapsto \sum_{x\in S}\operatorname{mult}_{x}(p)\cdot x,
\end{equation*}
where the right hand side is understood as a formal sum.
Note that \(p\in \mathrm{Sym}^{b}(S)\) represents a 
set of unlabelled \(b\) points on \(S\)
and labeling them is equivalent to 
choosing a preimage under the 
canonical surjection \(S^{b}\to \mathrm{Sym}^{b}(S)\).
Moreover, fixing a labeling and deforming \(p\) around gives rise to
a well-defined section of \(S^{b}\to \mathrm{Sym}^{b}(S)\)
as long as \(|\operatorname{Supp}(p)|=b\) remains constant
in the deformation.

Regard \(\mathcal{U}\) as a subscheme in \(S\times
\mathrm{Hilb}^{b}(S)\) and let \(\mathcal{Y}\) be
the blow-up of \(S\times
\mathrm{Hilb}^{b}(S)\)
along \(\mathcal{U}\). 
Then the general fiber of the family \(\mathcal{Y}\to 
\mathrm{Hilb}^{b}(S)\)
is the blow-up of \(S\) along distinct \(b\) points. 
We consider a codimension two closed subscheme
\begin{equation*}
T:=\{p\in\mathrm{Hilb}^{b}(S)~|~|\operatorname{Supp}(p)|\le b-1\}.
\end{equation*}

We can choose a curve \(C\) in \(\mathrm{Hilb}^{b}(S)\)
such that 
\begin{itemize}
	\item \(C\) meets \(T\) transversely and smooth at \(p\);
	\item \(C\) maps isomorphically onto its image 
	under the Hilbert--Chow morphism;
	\item any \(q\in C\setminus \{p\}\) near \(p\)
	represents a set of points in 
	almost general position.
\end{itemize}
Let \(\mathcal{U}\to C\) be the pullback
of the universal family \(\mathcal{Y}\to\mathrm{Hilb}^{b}(S)\).
We may regard \(\mathcal{U}\) as a (reducible) subscheme of \(S\times C\).
Let \(\mathcal{Z}\) be the blow-up of \(S\times C\)
along \(\mathcal{U}\) and \(\mathcal{Z}\to C\)
be the associated family. Note that \(\mathcal{Z}\)
is not smooth; it acquires an ordinary double
point singularity over \(p\in C\).

\begin{rk}
	\label{rk:local-model}
	One can construct the local model in the following way.
	Consider \(\mathbb{C}^{3}\) with coordinate \((x,y,t)\).
	The ideals \(\langle y,x-t\rangle\) and \(\langle y,x+t\rangle\)
	give two lines in \(\mathbb{C}^{3}\)
	whose union is defined by \(\langle y,x^{2}-t^{2}\rangle\).
	
	Denote by \(X\) the blow-up of \(\mathbb{C}^{3}\)
	along the ideal \(\langle y,x^{2}-t^{2}\rangle\);
	\begin{equation*}
	X = \operatorname{Proj} \mathbb{C}[x,y,t][\xi,\eta]\slash 
	\langle \xi y - \eta (x^{2}-t^{2})\rangle
	\end{equation*}
	where \(\operatorname{Proj}\) is taken with
	respect to the \(\mathbb{Z}\)-grading on \(\xi,\eta\)
	with \(\deg(\xi)=\deg(\eta)=1\).
	On the affine chart \(\xi\ne 0\), \(X\)
	is isomorphic to 
	\begin{equation*}
	\operatorname{Spec}\mathbb{C}[x,y,t,\eta']\slash
	\langle y - \eta'(x^{2}-t^{2})\rangle,~\eta'=\eta\slash \xi,
	\end{equation*}
	which is smooth, while on the affine chart \(\eta\ne 0\),
	\(X\) is isomorphic to 
	\begin{equation*}
	\operatorname{Spec}\mathbb{C}[x,y,t,\xi']\slash
	\langle \xi'y - (x^{2}-t^{2})\rangle,~\xi'=\xi\slash \eta,
	\end{equation*}
	which is singular and has a ODP singularity.
	Introduce an \(\mathbb{Z}_{2}\)-action on \(t\) via
	\begin{equation*}
	\mu\cdot t := t^{2},~\mbox{where \(\mu\) is the generator of \(\mathbb{Z}_{2}\)},
	\end{equation*}
	and denote by \(s=t^{2}\) the \(\mathbb{Z}_{2}\)-invariant coordinate. Then
	the quotient defines local model of a smoothing
	of an ordinary double point (at the origin on the affine chart \(\eta\ne 0\)) 
	\begin{equation*}
	\operatorname{Spec}\mathbb{C}[x,y,s,\xi']\slash
	\langle \xi'y - (x^{2}-s)\rangle\to
	\operatorname{Spec}\mathbb{C}[s].
	\end{equation*}
	This is the only affine chart of the local model of \(\mathcal{U}\to C\)
	containing the singularity of the singular fiber
	with \(p\) identifying with \(s=0\).
\end{rk}
We can compute the monodromy of \(\mathcal{Z}\to C\) around \(p\)
by Picard--Lefschetz formula.
Pick a smooth reference fibre \(\mathcal{Z}_{q}\) of \(\mathcal{Z}\to C\)
and denote by \(E_{1},\ldots,E_{b}\) the
exceptional divisors in the blow-up \(\mathcal{Z}_{q}\to S\)
such that \(x_{1}\) and \(x_{2}\) (the images of \(E_{1}\) and 
\(E_{2}\) on \(S\)) collides when \(q\mapsto p\).
Then we have
\begin{equation*}
\mathrm{H}^{2}(\mathcal{Z}_{q},\mathbb{Z})=
\mathrm{H}^{2}(S,\mathbb{Z})\oplus\mathbb{Z}
\langle E_{1},\ldots,E_{b}\rangle.
\end{equation*}
As before, the pullback is omitted. 
Note that \(E_{1}-E_{2}\) is a generator of the vanishing cohomology.
Then Picard--Lefschetz formula says that
the monodromy transformation \(\varpi\colon \mathrm{H}^{2}(\mathcal{Z}_{q},\mathbb{Z})
\to \mathrm{H}^{2}(\mathcal{Z}_{q},\mathbb{Z})\) is given by
\begin{equation*}
\varpi(\gamma) = \gamma + \langle \gamma, E_{1}-E_{2}\rangle (E_{1}-E_{2}),~
\gamma\in\mathrm{H}^{2}(\mathcal{Z}_{q},\mathbb{Z}).
\end{equation*}
This is on the nose the reflection on 
\(\mathrm{H}^{2}(\mathcal{Z}_{q},\mathbb{Z})\)
generated by the root \(E_{1}-E_{2}\).

Recall that rational elliptic surfaces are 
rational surfaces with an elliptic fibration structure admitting a section.
We will need the following proposition.
\begin{prop}
\label{prop:monodromy}
Let \((Y,D)\) be either a pair of a weak del Pezzo surface
and a smooth anti-canonical divisor or a rational elliptic 
surface and an anti-canonical divisor with configuration
\(\mathrm{II}\), \(\mathrm{III}\), \(\mathrm{IV}\),
\(\mathrm{IV}^{\ast}\), \(\mathrm{III}^{\ast}\),
\(\mathrm{II}^{\ast}\), or \(\mathrm{I}_{k}^{\ast}\) with \(k=0,\ldots,4\).
Let \(C\) be a smooth holomorphic curve in \(X:=Y\setminus D\) with \([C]^{2}=-2\).
Then the root reflection associated with \([C]\) on \(\mathrm{H}^{2}(X,\mathbb{Z})\)
can be realized as a monodromy transformation of some deformation of \((Y,D)\).
\end{prop}

\begin{proof}
	Let \(Y\) be a weak del Pezzo surface of degree \(d\)
	and \(D\in |-K_{Y}|\) be smooth;
	\((Y,D)\) is a blowup of \(\mathbf{P}^{2}\) along \(b=9-d\) points
	on a smooth cubic in \(\mathbf{P}^{2}\). Denote by 
	\(E_{1},\ldots,E_{b}\) the pullback of 
	the exceptional divisors, \(x_{1},\ldots,x_{b}\)
	be the corresponding points on \(\mathbf{P}^{2}\)
	and by \(H\)
	the pullback of the hyperplane class on \(\mathbf{P}^{2}\).
	By \cite{MS22}*{Lemma 2.8}, \(C\) is given by
	\begin{itemize}
		\item[(1)] the proper transform of \(E_{i}\) over which
		there exists exactly one \(E_{j}\) lying over;
		\item[(2)] the proper transform of a line in \(\mathbf{P}^{2}\)
		passing through exactly three points in \(\{x_{1},\ldots,x_{b}\}\);
		\item[(3)] the proper transform of a conic in \(\mathbf{P}^{2}\)
		passing through exactly six points in \(\{x_{1},\ldots,x_{b}\}\);
		\item[(4)] the proper transform of a cubic in \(\mathbf{P}^{2}\)
		passing through exactly eight points in \(\{x_{1},\ldots,x_{b}\}\)
		such that one of which is the singular point of the cubic.
	\end{itemize}
	Case (1) occurs when \(b\ge 2\), Case (2)
	occurs when \(b\ge 3\), Case (3) occurs when \(b\ge 6\) and Case
	(4) appears only when \(b=8\).

	The root reflection from (1) can be realized 
	by collapsing \(x_{i}\) and \(x_{j}\).
	For (2), we begin with \(\mathbf{P}^{2}\) and
	pick a line \(H\) joining \(x_{i}\) and \(x_{j}\).
	Consider \(F:=H-E_{i}-E_{j}\) (the proper transform of \(H\) in \(Y\)).
	Then \(F\) is a \((-1)\) curve on the smooth surface \(Y\).
	By Castelnuovo's theorem, we can contract \(F\) to \(x\in Y'\) for
	a smooth surface \(Y'\). Since \(x_{k}\notin H\), \(x_{k}\)
	is mapped to a point \(x_{k}'\in Y'\). Note that the class \(H-E_{i}-E_{j}-E_{k}\)
	is equal to \(F-E_{k}\) in \(\mathrm{H}^{2}(Y,\mathbb{Z})\).
	Regarding \(E_{k}'\) as the exceptional divisor over \(x_{k}'\),
	we see that the associated root reflection can be realized as the
	monodromy transformation of the degeneration by collapsing \(x_{k}'\) and \(x\).
	The remaining cases can be treated in a similar way.

	Let \(Y\) be a rational elliptic 
	surface and \(D\) an anti-canonical divisor with configuration
	described in the Proposition. It is known that \(Y\)
	is a blow-up of the base locus of a pencil of cubics
	on \(\mathbf{P}^{2}\) with a smooth member.
	Let \(\pi\colon Y\to\mathbf{P}^{1}\) be the 
	associated elliptic fibration. We also assume that
	\(D\) is the fiber at \(\infty\in\mathbf{P}^{1}\).
  Let \(C\) be as in the proposition.
	Then \(C\) must be an irreducible component of
	a fiber of \(\pi\). There exists a sequence of 
	blow-downs \((-1)\) curves
	\begin{equation*}
	Y=Z_{0}\to Z_{1}\to\cdots\to Z_{k}=:Z
	\end{equation*}
	satisfying the following properties
	\begin{itemize}
	\item the image of \(C\) in \(Z_{k-2}\) (still denoted by \(C\)) remains a \((-2)\) curve;
	\item \(Z_{k-2}\to Z_{k-1}\) is a contraction of 
	a \((-1)\) curve \(F\) with \(F\cap C\ne\emptyset\),
	and therefore \(C\) becomes a \((-1)\) curve in \(Z_{k-1}\);
	\item \(Z_{k-1}\to Z\) is
	given by contracting \(C\). 
	\end{itemize}

	Moreover, by Castelnuovo's theorem,
	\(Z_{i}\) is a smooth projective variety for each \(i=0,\ldots,k\).
	(Indeed, one can begin with contracting a section of \(\pi\).
	Since any section must meet the fiber containing \(C\) and 
	every fiber is connected, one can continue the process to reach \(C\).)
	Using \(\mathrm{Hilb}^{2}(Z)\) the Hilbert scheme of length \(2\)
	subscheme on \(Z\),
	from the discussion right before Proposition \ref{prop:monodromy},
	we can find a suitable degeneration
	whose monodromy transformation equals the 
	root reflection constructed from \([C]\).
	This completes the proof.
\end{proof}

\subsection{Surjectivity of period maps of \texorpdfstring{$\mathrm{ALH}^{\ast}$}{ALH*} gravitational instantons}
\label{subsec:alh*-surj}
   Any holomorphic curve in $X$ is a $(-2)$-curve in $Y$ by adjuction formula. 
	We first recall a theorem of Tian--Yau \cite{TY}. We say a cohomology class $[\omega]\in \mathrm{H}^2(X,\mathbb{R})$ 
	satisfies the condition ($\dagger$) if $[\omega]$ is positive on every $(-2)$-curve of $Y$ contained in $X$ and there exists a K\"ahler class $[\omega_Y]$ on $Y$ such that $[\omega]=[\omega_Y]|_X$.
	
	\begin{thm}\label{TY} 
		Given $c>0$ and $[\omega]\in \mathrm{H}^2(X,\mathbb{R})$ satisfies the condition {\rm($\dagger$)}, then there exists a Ricci-flat metric $\omega$ in the given cohomology class on $X$ with
			 $2\omega^2= \Omega\wedge\bar{\Omega}$,
		 where $\Omega$ is a meromorphic volume form on $Y$ with a simple pole along $D$ such that $\operatorname{Res}_D\Omega=c\Omega_D$. 
	\end{thm}
 \begin{proof}[Proof of Theorem \ref{main1}]
  	From Theorem \ref{thm:surjectivity-of-periods}, there exists a del Pezzo surface $Y$ of degree $d$ and a smooth anti-canonical divisor $D$ with modulus $\tau$ and a meromorphic volume form $\Omega$ on $Y$ with a simple pole along $D$ and $\operatorname{Res}_D\Omega=c\Omega_D$ such that there exists a diffeomorphism $\mu\colon X_0\rightarrow X$ with $\mu^*[\Omega]=[\Omega_0]$, where $X=Y\setminus D$. 
 	
 	If $(\mu^{-1})^*[\omega_0]$ satisfies the condition ($\dagger$), then the theorem follows from Theorem \ref{TY} directly. Otherwise, from Proposition \ref{prop:monodromy} and \cite{Dol}*{Theorem 2.1} there exists a diffeomorphism $g\colon X\rightarrow X$ such that $g^*(\mu^{-1})^*[\omega_0]$ satisfies the condition ($\dagger$). Thus, there exists a Ricci-flat metric $\omega$ in the cohomology class $g^*(\mu^{-1})^*[\omega_0]$ by Theorem \ref{TY}. Notice that the diffeomorphism is induced by the compositions of monodromies in the moduli space of pairs $(Y,D)$. In particular, the Picard--Lefchetz formula implies that $g^*[\Omega]=[\Omega]$ since $[\Omega]$ vanishes on every $(-2)$-curve in $X$. Then $(X,\omega ,\Omega,g^{-1}\circ \mu)$ is the marked $\mathrm{ALH}^{\ast}_d$ gravitational instantons realizing the given cohomology classes and finish the proof of the theorem. 
 	
 \end{proof}  

\section{Period domains for 
\texorpdfstring{$\mathrm{ALG}$}{ALG} and
\texorpdfstring{$\mathrm{ALG}^{\ast}$}{ALG*} gravitational instantons}
\label{sec:alg}
Recall that 
rational elliptic surfaces are rational surfaces with an elliptic fibration structure\footnote{Here we use the definition that an elliptic fibration admits a section.}.
It is well known that any rational elliptic surface is a blow up of the base points of a pencil of cubics with a smooth member in $\mathbf{P}^2$, i.e., given a rational elliptic surface $Y$ and a fibre $D$, the pair $(Y,D)$ can be derived from blow-ups of $\mathbf{P}^2$ on a possibly singular cubic $\underline{D}$ and $D$ contains the proper transform of $\underline{D}$. 

         
\begin{defn}
An \emph{ALG pair} is a log Calabi--Yau pair \((Y,D)\) with 
\(Y\) a smooth rational elliptic surface and 
\(D\in |-K_{Y}|\) a divisor of type
$\mathrm{II},
\mathrm{III},\mathrm{IV},\mathrm{IV}^{\ast},
\mathrm{III}^{\ast}$, $\mathrm{II}^{\ast}$, or \(\mathrm{I}_{0}^{\ast}\).
An \emph{ALG pair of type \(\mathrm{Z}\)} is an ALG pair \((Y,D)\)
such that \(D\) is of type \(\mathrm{Z}\).
A \emph{marked ALG pair} is an ALG pair \((Y,D)\)
together with a basis \(\mathcal{B}\) of \(\mathrm{H}_{2}(X,\mathbb{Z})\).
Finally, a \emph{marked ALG pair of type \(\mathrm{Z}\)} is a marked ALG pair
of type \(\mathrm{Z}\) with a basis \(\mathcal{B}\) of \(\mathrm{H}_{2}(X,\mathbb{Z})\).

Similarly, we can define the notions for \(\mathrm{ALG}^{\ast}\) pairs. In which case,
the configurations of \(D\) can be \(\mathrm{I}_{1}^{\ast}\), 
\(\mathrm{I}_{2}^{\ast}\), \(\mathrm{I}_{3}^{\ast}\), and
\(\mathrm{I}_{4}^{\ast}\).
\end{defn}
         
\subsection{Constructions of \texorpdfstring{\((Y,D)\)}{(Y,D)} for \texorpdfstring{\(\mathrm{ALG}\)}{ALG} and 
\texorpdfstring{\(\mathrm{ALG}^{\ast}\)}{ALG*} gravitational instantons}
\label{subsec:construction}
In this subsection, we will give constructions 
of families of marked ALG and \(\mathrm{ALG}^{\ast}\) 
pairs of various types and study their period maps. 

We begin with a general discussion.
Let \((Y,D)\) be either an ALG pair or an \(\mathrm{ALG}^{\ast}\) pair
and \(X=Y\setminus D\) be the complement. In any case,
we have
\(\mathrm{H}^{1}(D,\mathbb{C})=0\).
 Let \(i\colon D\to Y\) be the closed embedding 
 and \(j\colon X\to Y\) be the open embedding.
 We have the short exact sequence
 \begin{equation}
 0\to j_{!}j^{-1}\mathbb{Q} \to \mathbb{Q} \to i_{\ast} i^{-1}\mathbb{Q}\to 0.
 \end{equation}
 Taking compactly supported cohomology yields the 
 long exact sequence
 \begin{equation}
 \label{eq:alg-alg*-es}
 \cdots\to 0=\mathrm{H}_{\mathrm{c}}^{1}(D,\mathbb{Q})\to 
 \mathrm{H}_{\mathrm{c}}^{2}(X,\mathbb{Q})\to 
 \mathrm{H}_{\mathrm{c}}^{2}(Y,\mathbb{Q})\to
 \mathrm{H}_{\mathrm{c}}^{2}(D,\mathbb{Q})\to\cdots.
 \end{equation}
 It then follows that the homology group \(\mathrm{H}_{2}(X,\mathbb{Q})
 \cong \mathrm{H}_{\mathrm{c}}^{2}(X,\mathbb{Q})\)
 can be identified with the kernel
 of the signed intersection map
 \begin{equation}
\gamma\mapsto (\gamma\cdot D_{i})_{i=1}^{k}
 \end{equation}
 where \(D=\sum_{i=1}^{k}m_{i}D_{i}\) and \(D_{i}\)'s are irreducible components.
 By the vanishing of \(\mathrm{H}^{1}(D,\mathbb{C})\),
 any \(\gamma\in \mathrm{H}_{2}(Y,\mathbb{C})\)
 satisfying \(\gamma\cdot D_{i} = 0\) for all \(i\) can be lifted
 \emph{uniquely} to \(\mathrm{H}_{2}(X,\mathbb{C})\).
 
 To construct a basis \(\mathcal{B}\), we can therefore
 pick any basis of 
 \begin{equation*}
 \{\gamma\in \mathrm{H}_{2}(Y,\mathbb{Z})~|~\gamma\cdot D_{i}=0,~i=1,\ldots,k\}
 \end{equation*}
 and lift it to \(\mathrm{H}_{2}(X,\mathbb{Z})\).
 
 Let \(\pi\colon (\mathcal{Y},\mathcal{D})\to\mathcal{M}\) be a deformation family of a marked ALG pair \((Y,D)\)
 and \(\mathfrak{r}\in\mathcal{M}\) be a reference pair. 
 We denote the reference pair by \((Y_{\mathfrak{r}},D_{\mathfrak{r}})\)
 and its complement by \(X_{\mathfrak{r}}:=Y_{\mathfrak{r}}\setminus D_{\mathfrak{r}}\). 
 Let \(\Omega\) be a section of \(\pi_{\ast}\Omega_{\mathcal{Y}\slash\mathcal{M}}^{2}(\mathcal{D})\).
 Assuming \(\mathcal{M}\) is simply connected, we can define
 the period integrals of \((Y,D)\) to be the function
 \begin{equation}
 \mathcal{M}\ni t\to \int_{(\varphi_{\Gamma})_{\ast}\gamma}\Omega_{t}
 \end{equation}
 where \(\varphi_{\Gamma}\colon X_{\mathfrak{r}}
 \to X\) is the diffeomorphism induced
 by a path \(\Gamma\) connecting \(t\) and 
 \(\mathfrak{r}\) in \(\mathcal{M}\),
 \(\gamma\in \mathrm{H}_{2}(X_{\mathfrak{r}},\mathbb{Z})\),~
 and \(\Omega_{t}\)
 is the restriction of \(\Omega\) to the fibre \((Y_{t},D_{t})\).
 This is well-defined since \(\mathcal{M}\) is simply connected.
 When \(\pi_{1}(\mathcal{M})\) is non-trivial, the period integrals above form a local system
 on \(\mathcal{M}\) and in general have non-trivial monodromies.
 
 To define period integrals in general cases, 
 we must keep track of the trivialization, i.e.,
 the trivialization from the curve \(\Gamma\).
\begin{defn}
Let \(\pi\colon (\mathcal{Y},\mathcal{D})\to\mathcal{M}\) be a deformation family of a marked ALG 
or or \(\mathrm{ALG}^{\ast}\) pair \((Y,D)\)
and \(\mathfrak{r}\in\mathcal{M}\) be a reference pair. 
Let \(\Omega\) be a section of \(\pi_{\ast}\Omega_{\mathcal{Y}\slash\mathcal{M}}^{2}(\mathcal{D})\).
For \(t\in\mathcal{M}\),
the period integral is defined to be the multi-valued function
\begin{equation}
\mathcal{M}\ni t\mapsto\int_{(\varphi_{\Gamma})_{\ast}\gamma}\Omega_{t}
\end{equation}
where \(\varphi_{\Gamma}\colon X_{\mathfrak{r}}
 \to X\) is the diffeomorphism induced
 by a path \(\Gamma\) connecting \(t\) and 
 \(\mathfrak{r}\) in \(\mathcal{M}\),
 \(\gamma\in \mathrm{H}_{2}(X_{\mathfrak{r}},\mathbb{Z})\),~
 and \(\Omega_{t}\)
 is the restriction of \(\Omega\) to the fibre \((Y_{t},D_{t})\).
 Let \(\mathcal{B}_{\mathfrak{r}}:=\{\gamma_{\mathfrak{r},1},
 \ldots,\gamma_{\mathfrak{r},m}\}\) be
 a basis of \(\mathrm{H}_{2}(X_{\mathfrak{r}},\mathbb{Z})\). 
 The \emph{period map} \(\mathcal{P}_{\mathcal{B}_{\mathfrak{r}}}(\Omega)\) is
 a multi-vector-valued function 
 \begin{equation}
 \mathcal{M}\ni t\mapsto \left(
 \int_{(\varphi_{\Gamma})_{\ast}\gamma_{\mathfrak{r},1}}\Omega_{t},\ldots,
 \int_{(\varphi_{\Gamma})_{\ast}\gamma_{\mathfrak{r},m}}\Omega_{t}\right)
\in\mathbb{C}^{m}.
 \end{equation}
 For simplicity, when the context is clear, we drop 
 \(\mathcal{B}_{\mathfrak{r}}\) and \(\Omega\) in the notation.
\end{defn}

Note that \(\operatorname{Im}(\mathcal{P})\) always lies in a hyperplane in \(\mathbb{C}^{m}\)
determined by the fibre class. More precisely, let \(\mathcal{B}_{\mathfrak{r}}
=\{\gamma_{\mathfrak{r},1},\ldots,\gamma_{\mathfrak{r},m}\}\)
and assume that 
\begin{equation*}
[f] = \sum_{i=1}^{m} a_{i}\gamma_{\mathfrak{r},i}
\end{equation*}
where \([f]\) is the homology class of a fibre
in the rational elliptic surface \(Y_{\mathfrak{r}}\). Then 
\begin{equation*}
\sum_{i=1}^{m}a_{i}\int_{\gamma_{\mathfrak{r},i}}\Omega_{t} = 0.
\end{equation*}
That is, if \((y_{1},\ldots,y_{m})\) denotes the coordinates on \(\mathbb{C}^{m}\),
we have
\begin{equation*}
\operatorname{Im}(\mathcal{P})\subset 
\left\{(y_{1},\ldots,y_{m})~\Big|~\sum_{i=1}^{m} a_{i}y_{i}=0\right\}.
\end{equation*}
The main result in this subsection is the following theorem.
\begin{thm}\label{thm: surj (2,0)-form}
Let notation be as above. 
Let \((Y,D\)) be
an \(\mathrm{ALG}\) or an 
\(\mathrm{ALG}^{\ast}\) pair. Then
there exist a family \(\pi\colon (\mathcal{Y},\mathcal{D})\to
\mathcal{M}\) of deformation of \((Y,D)\) and \(\Omega\in\Omega_{\mathcal{Y}\slash\mathcal{M}}^{2}(\mathcal{D})\) 
such that 
\begin{equation}
\label{eq:period-map-surjective}
\operatorname{Im}(\mathcal{P}) =
\left\{(y_{1},\ldots,y_{m})~\Big|~\sum_{i=1}^{m} a_{i}y_{i}=0\right\}.
\end{equation}  

\end{thm}
The rest of the subsection is devoted to 
proving Theorem \ref{thm: surj (2,0)-form}. 
To achieve this, we will 
\begin{itemize}
\item construct for each type 
a reference marked ALG or \(\mathrm{ALG}^{\ast}\) pair \((Y_{\mathfrak{r}},D_{\mathfrak{r}})\), i.e.,
\begin{itemize}
\item a pencil of cubics
in \(\mathbf{P}^{2}\) giving the ALG or \(\mathrm{ALG}^{\ast}\) pair \((Y_{\mathfrak{r}},D_{\mathfrak{r}})\) 
of the desired type after resolving 
the base locus;
\item a basis \(\mathcal{B}_{\mathfrak{r}}\) of \(\mathrm{H}_{2}(X_{\mathfrak{r}},\mathbb{Z})\) 
where \(X_{\mathfrak{r}}:=Y_{\mathfrak{r}}\setminus D_{\mathfrak{r}}\)
is the complement;
\item a choice of a section \(\Omega\in \pi_{\ast}\Omega_{\mathcal{Y}\slash\mathcal{M}}^{2}(\mathcal{D})\);
\end{itemize}
\item analyze \(\operatorname{Im}(\mathcal{P})\)
the image of the period map defined by the data in the first bullet and
argue that the equality \eqref{eq:period-map-surjective} holds.
\end{itemize}
To achieve these, one also needs to verify
that the pencil constructed in various situations 
contains a smooth member.
The following two classical results will be useful.
First, we recall Bertini's theorem. 
\begin{thm}[Bertini's Theorem~\cite{GH}*{p.~137}]
\label{thm:bertini}
Assume the ambient variety is smooth. 
Then general elements of a pencil 
are smooth away from its base locus.
\end{thm}
Assume that the pencil is spanned by \(C\) and \(D\). 
Suppose that
\(C\cap D\in C_{\mathrm{sm}}\cup D_{\mathrm{sm}}\),
that is, the intersection points \(C\cap D\)
are either a smooth point of \(C\) or a smooth point of \(D\).
Then the linear system \(|uC+vD|\) contains a smooth member. Indeed, by
Bertini's theorem,
one could pick a general member \(A\) which is smooth outside the base locus.
Since \(C\cap D\in C_{\mathrm{sm}}\cup D_{\mathrm{sm}}\), we can 
perturb the defining equation of \(A\) to 
eliminate the singularities of \(A\) by adding the defining
equation of \(C\) or \(D\).

Second, we recall the adjunction formula and the residue formalism
which will be crucial in our calculation of periods.
%
%
Let \(Y\) be a smooth algebraic variety over \(\mathbb{C}\)
and \(D\in |-K_{Y}|\) be an anti-canonical divisor.
We will be interested in the case when \(D\) is singular. 
More precisely, assume that \(D=\sum_{i=1}^{r} m_{i}D_{i}\in |-K_{Y}|\)
is anti-canonical such that each \(D_{i}\) is smooth
but the intersections are allowed to be non-transversal
(e.g.~three lines meet at one point in \(\mathbf{P}^{2}\)).

Let \(\Omega^{n}_{Y}(D)\) be the sheaf of meromorphic differentials
on \(Y\) whose pole divisor is equal to \(D\) where
\(n=\dim_{\mathbb{C}}Y\). (This is indeed 
a trivial bundle owing to our assumption \(D\in |-K_{Y}|\).) Then if \(m_{1}=1\), 
from the adjunction formula, we have 
\begin{prop} 
\label{prop:residue}
\begin{equation}
\Omega_{Y}^{n}(D)|_{D_{1}} = 
\Omega_{Y}^{n}(D_{1})|_{D_{1}}\otimes\mathcal{O}_{Y}(D-D_{1})|_{D_{1}}
\xrightarrow{\cong}\Omega_{D_{1}}^{n-1}((D-D_{1})|_{D_{1}}).
\end{equation}
This isomorphism is realized by the Poincar\'{e} residue.
\end{prop}

\subsubsection{{\bf Type \(\mathrm{II}\)}}
\label{subsubsec:type-ii-construction}
    A type \(\mathrm{II}\) fibre is a rational curve with a cusp singularity. 
    In this case, $\underline{D}$ is a cuspidal rational curve
    in \(\mathbf{P}^{2}\) with the cusp at $p$ 
    and $Y$ is the blow up of \(\mathbf{P}^{2}\) 
    at nine points on $\underline{D}\setminus \{p\}$.  
    It is well-known that the cuspidal rational curve 
    in $\mathbf{P}^2$ is unique up to 
    $\mathrm{PGL}(3)$-action \cite{Ful}*{p.~55}.
    We may assume that \(\underline{D}\) 
    is given by the equation \(\{y^{2}z-x^{3}=0\}\).
    Recall that the Cayley--Bacharach theorem states that any distinct eight 
    points in $\mathbf{P}^2$ without four on a line or seven points on a non-degenerate 
    conic would determine uniquely a pencil of cubics; 
    in other words, any such eight points 
    in $\underline{D}\setminus\{p\}$ determines a pencil of cubics. 
    Since the other members in the pencil avoid $p$, 
    the pencil must contain at least one smooth member and thus determines a 
    rational elliptic surface. 
    We also remark that \(\underline{D}\setminus \{p\}\) is an affine group variety
    which is isomorphic to \(\mathbb{C}\) (the additive group).
    
      To construct a marked reference ALG pair of type \(\mathrm{II}\),
  we simply pick a smooth cubic \(C\) which meets \(\underline{D}\setminus\{p\}\) at nine points.
  Denote by \(p_{\mathfrak{r},1},\ldots,p_{\mathfrak{r},9}\) 
  the intersections \(\underline{D}\cap C\). Let
  \(Y_{\mathfrak{r}}\) be the blow-ups of \(\mathbf{P}^{2}\) at those nine points
  and \(D_{\mathfrak{r}}\) be the proper transform of \(\underline{D}\).
  Then \((Y_{\mathfrak{r}},D_{\mathfrak{r}})\) is an ALG pair
  of type \(\mathrm{II}\) (cf.~figure (a) in \textsc{Figure}~\ref{fig:alg-pair}). 
  Let \(E_{\mathfrak{r},1},\ldots,
  E_{\mathfrak{r},9}\) be exceptional divisors. Then
  \begin{equation}
  \mathcal{B}_{\mathfrak{r}}:=\{H_{\mathfrak{r}}-3E_{\mathfrak{r},1}\}
  \bigcup\cup_{i=2}^{9} \{E_{\mathfrak{r},1}-E_{\mathfrak{r},i}\}
  \end{equation}
  is a basis of \(\mathrm{H}_{2}(X_{\mathfrak{r}},\mathbb{C})\).
  For simplicity, we denote the elements in \(\mathcal{B}_{\mathfrak{r}}\)
  by \(\gamma_{\mathfrak{r},1},\ldots,\gamma_{\mathfrak{r},9}\).
  The fibre class is represented by 
  \begin{equation*}
  3\gamma_{\mathfrak{r},1} + \sum_{j=2}^{9} \gamma_{\mathfrak{r},j}.
  \end{equation*}
  Let 
  \begin{equation}
  \label{eq:meromorphic-form-ii}
  \Omega = \frac{x\mathrm{d}y\wedge\mathrm{d}z-
  y\mathrm{d}x\wedge\mathrm{d}z+z\mathrm{d}x\wedge\mathrm{d}y}{y^{2}z-x^{3}}.
  \end{equation}
  It follows that 
  \begin{equation}
  \operatorname{Im}(\mathcal{P})\subset 
  \left\{(y_{1},\ldots,y_{9})~\Big|~3y_{1}+\sum_{j=2}^{9} y_{j}=0\right\}.
  \end{equation}
  Under the affine coordinates \(u=x/y\) and \(v=z/y\), we have
  \begin{equation*}
  \Omega = -\frac{\mathrm{d}u\wedge\mathrm{d}v}{v-u^{3}}
  \end{equation*}
  and the residue around \(\{v-u^{3}=0\}\) is \(-\mathrm{d}u\).

  Now we prove that the inclusion above is indeed an equality.
  Let \((y_{1},\ldots,y_{9})\) be a vector satisfying
  the condition \(3y_{1}+\sum_{j=2}^{9} y_{j}=0\).
  We will need a few computational results.
  \begin{lem}
  \label{lem:period-type-ii}
  Let \(\underline{D}=\{y^{2}z-x^{3}=0\}\subset\mathbf{P}^{2}\)
  and \(p=[0\mathpunct{:}0\mathpunct{:}1]\)
  be the unique singular point on \(\underline{D}\). Let
  \(\Omega\) be the meromorphic two form defined in \eqref{eq:meromorphic-form-ii}.
  Let \(x_{1},\ldots,x_{9}\in \underline{D}\setminus \{p\}\cong\mathbb{C}\) and 
  \(Y\) is the blow-up of \(\mathbf{P}^{2}\) at \(x_{1},\ldots,x_{9}\).
  Denote by \(E_{i}\) the exceptional divisor over \(x_{i}\).
  Then
  \begin{itemize}
  \item[(a)] Let \(H\) be the hyperplane class in \(\mathbf{P}^{2}\). 
  We have
  \begin{equation*}
  \int_{H-3E_{1}}\Omega = -3x_{1}.
  \end{equation*}
  \item[(b)] 
  \begin{equation*}
  \int_{E_{i}-E_{j}}\Omega = x_{i}-x_{j}.
  \end{equation*}
  \end{itemize}
  \end{lem}
  \begin{proof}
  This follows from the residue calculations.
  \end{proof}

  By Lemma \ref{lem:period-type-ii}, the vector \((y_{1},\ldots,y_{9})\) 
  uniquely determines the points \(x_{1},\ldots,x_{9}\) on \(\underline{D}\).
  Moreover, the constraint \(3y_{1}+\sum_{j=2}^{9} y_{j}=0\)
  implies that \(x_{1}+\cdots+x_{9}=O\) in the 
  additive group scheme \(\underline{D}\setminus\{p\}\) and
  it turns out that this condition is sufficient by Max Noether's fundamental
  theorem \cite{Ful}*{p.~61},
  i.e., given any 9 points 
  \(x_{1},\ldots,x_{9}\in \underline{D}\setminus\{p\}\) (not necessarily distinct) with
  \(x_{1}+\cdots+x_{9}=O\), there exists a cubic \(C\)
  passing through all the \(x_{i}\)'s.
  According to Theorem \ref{thm:bertini}
  and the discussion after it,
  the pencil spanned by \(C\) and \(\underline{D}\) contains a smooth member.
  This shows that \(Y=\mathrm{Bl}_{\{x_{1},\ldots,x_{9}\}}\mathbf{P}^{2}\)
  is a smooth rational elliptic surface.


\subsubsection{{\bf Type $\mathrm{III}$}}
\label{subsubsec:type-iii-construction}
    A type \textrm{III}
    fibre is a union of three smooth rational curves intersecting at a single point. 
    In this case, $\underline{D}$ can be three lines $L_1\cup L_2\cup L_3$ or $C\cup L$, where $L$ is a line and $C$ is a conic tangent to $L$ at $p$. In the former case, each $L_i$ contains three points of the blow up loci. In the latter case, $Y$ is the blow up of five points on $C\setminus \{p\}$, two points on $L\setminus \{p\}$ and $p$ then blow up a point on the exceptional curve corresponding to $p$ avoiding the proper transform of $C$. 

  To construct a marked reference ALG pair of type \(\mathrm{III}\),
  we pick a smooth cubic \(C\) which meets \(\underline{D}=L_{1}\cup L_{2}\cup L_{3}\) at nine points.
  Denote by \(p_{\mathfrak{r},1},\ldots,p_{\mathfrak{r},9}\) 
  the intersections \(\underline{D}\cap C\) in a way such that
  \(p_{\mathfrak{r},i}\in L_{j}\) if and only if \(i\equiv j\mod{3}\). Let
  \(Y_{\mathfrak{r}}\) be the blow-ups of \(\mathbf{P}^{2}\) at those nine points
  and \(D_{\mathfrak{r}}\) be the proper transform of \(\underline{D}\).
  Then \((Y_{\mathfrak{r}},D_{\mathfrak{r}})\) is an ALG pair
  of type \(\mathrm{III}\) (cf.~figure (b) in \textsc{Figure}~\ref{fig:alg-pair}). 
  Let \(E_{\mathfrak{r},1},\ldots,
  E_{\mathfrak{r},9}\) be exceptional divisors. Then
  \begin{align*}
  \mathcal{B}_{\mathfrak{r}}:=&\{H_{\mathfrak{r}}-E_{\mathfrak{r},1}
  -E_{\mathfrak{r},2}-E_{\mathfrak{r},3}\}
  \bigcup\cup_{i=4,7} \{E_{\mathfrak{r},1}-E_{\mathfrak{r},i}\}\\
  &\bigcup\cup_{i=5,8} \{E_{\mathfrak{r},2}-E_{\mathfrak{r},i}\}
  \bigcup\cup_{i=6,9} \{E_{\mathfrak{r},3}-E_{\mathfrak{r},i}\}
  \end{align*}
  is a basis of \(\mathrm{H}_{2}(X_{\mathfrak{r}},\mathbb{C})\).
  For simplicity, we denote the elements in \(\mathcal{B}_{\mathfrak{r}}\)
  by \(\gamma_{\mathfrak{r},1},\ldots,\gamma_{\mathfrak{r},7}\).
  The fibre class is represented by 
  \begin{equation*}
  3\gamma_{\mathfrak{r},1} + \sum_{j=2}^{7} \gamma_{\mathfrak{r},j}.
  \end{equation*}
  Let \(\underline{D}=\{xy(x+y)=0\}\). This can be always achieved using the
  \(\mathrm{PGL}(3,\mathbb{C})\) action. Let
  \begin{equation}
  \label{eq:meromorphic-form-iii}
  \Omega = \frac{x\mathrm{d}y\wedge\mathrm{d}z-
  y\mathrm{d}x\wedge\mathrm{d}z+z\mathrm{d}x\wedge\mathrm{d}y}{xy(x+y)}.
  \end{equation}
  It follows that 
  \begin{equation}
  \operatorname{Im}(\mathcal{P})\subset 
  \left\{(y_{1},\ldots,y_{7})~\Big|~3y_{1}+\sum_{j=2}^{7} y_{j}=0\right\}.
  \end{equation}
  Under the affine coordinates \(u=y/x\) and \(v=z/x\), we have
  \begin{equation*}
  \Omega = -\frac{\mathrm{d}u\wedge\mathrm{d}v}{u(1+u)}.
  \end{equation*}
  We will need the following computational results.
  \begin{lem}
  \label{lem:period-type-iii}
  Let \(\underline{D}=\{xy(x+y)=0\}\subset\mathbf{P}^{2}\)
  and \(p=[0\mathpunct{:}0\mathpunct{:}1]\)
  be the unique singular point on \(\underline{D}\). Let
  \(\Omega\) be the meromorphic two form defined in \eqref{eq:meromorphic-form-iii}.
  Let \(x_{1},\ldots,x_{9}\in \underline{D}\setminus \{p\}\) 
  such that \(x_{i}\in L_{j}\) if and only if \(i\equiv j\mod{3}\) and 
  \(Y\) be the blow-up of \(\mathbf{P}^{2}\) at \(x_{1},\ldots,x_{9}\).
  Denote by \(E_{i}\) the exceptional divisor over \(x_{i}\) as before.
  Let \(x_{i}=[0\mathpunct{:}a_{i}\mathpunct{:}b_{i}]\) for \(i=1,4,7\),
  \(x_{i}=[a_{i}\mathpunct{:}0\mathpunct{:}b_{i}]\) for \(i=2,5,8\),
  and \(x_{i}=[a_{i}\mathpunct{:}-a_{i}\mathpunct{:}b_{i}]\) for \(i=3,6,9\).
  Then
  \begin{itemize}
  \item[(a)] 
  The line \(\overline{x_{1}x_{2}}\) intersects \(x+y=0\)
  at \([a_{1}a_{2}\mathpunct{:}-a_{1}a_{2}\mathpunct{:}a_{1}b_{2}-a_{2}b_{1}]\ne p\).
  (Note that \(a_{i}\ne 0\) for all \(i\) by our assumption.)
  Let \(H\) be the hyperplane class in \(\mathbf{P}^{2}\). 
  Then we have
  \begin{equation*}
  \int_{\gamma_{1}}\Omega =
  \int_{H-E_{1}-E_{2}-E_{3}}\Omega = \frac{a_{1}b_{2}-a_{2}b_{1}}{a_{1}a_{2}}-
  \frac{b_{3}}{a_{3}}.
  \end{equation*}
  \item[(b)] We have
  \begin{equation*}
  \int_{E_{k+3}-E_{k}}\Omega = \frac{b_{k+3}}{a_{k+3}}-\frac{b_{k}}{a_{k}},~\mbox{and}~
  \int_{E_{k+6}-E_{k}}\Omega = \frac{b_{k+6}}{a_{k+6}}-\frac{b_{k}}{a_{k}}
  \end{equation*}
  for \(k=1,2,3\).
  \end{itemize}
  \end{lem}
  \begin{proof}
  This follows from a direct calculation on residues and hence the proof is omitted.
  \end{proof}

By Lemma \ref{lem:period-type-iii}, the vector \((y_{1},\ldots,y_{7})\in\mathbb{C}^{7}\)
determines \(x_{1},\ldots,x_{9}\) on \(\underline{D}\setminus\{p\}\).
Indeed, we can put \(x_{1}=p_{\mathfrak{r},1}\) and 
\(x_{2}=p_{\mathfrak{r},2}\) and the results in (a) and (b) in
Lemma \ref{lem:period-type-iii} would determine the location of all the rest \(x_{i}\)'s.
The only thing we have to show is that 
\(Y=\mathrm{Bl}_{\{x_{1},\ldots,x_{9}\}}\mathbf{P}^{2}\)
is a rational elliptic surface, i.e.,
there is a smooth cubic passing through \(x_{1},\ldots,x_{9}\).

Again it suffices to construct a cubic passing through
the points \(x_{1},\ldots,x_{9}\). This can be done directly. Indeed, 
suppose the coordinate of \(x_{i}\) is given as
in Lemma \ref{lem:period-type-iii}. The cubic defined by 
\begin{equation}
\prod_{i=1,4,7} (b_{i}y-a_{i}z) + x\cdot (ax^{2} + by^2 + cz^{2} + dxy + eyz + fxz)=0
\end{equation}
passing through \(x_{1},x_{4},x_{7}\). Now set \(y=0\) in 
the above equation. We obtain
\begin{equation}
-a_{1}a_{4}a_{7}z^{3} + ax^{3}+cxz^{2}+fx^{2}z=0=
\frac{a_{1}a_{4}a_{7}}{a_{2}a_{5}a_{8}}\prod_{i=2,5,8} (b_{i}x-a_{i}z).
\end{equation}
This equation uniquely determines the coefficients \(a\), \(c\), and \(f\).
We are left with \(b\), \(d\) and \(e\),
i.e., the coefficient of \(xy^{2}\), \(x^{2}y\), and \(xyz\).
Now set \(y=-x\) in the above equation. We see that
\begin{align*}
&-\prod_{i=1,4,7} (b_{i}x+a_{i}z) + x\cdot (ax^{2} + bx^2 + cz^{2} - dx^{2} - exz + fxz)=0\\
&=-\frac{a_{1}a_{4}a_{7}}{a_{3}a_{6}a_{9}}\prod_{i=3,6,9}(a_{i}z-b_{i}x)
\end{align*}
from which \(e\) and \(b-d\) are uniquely determined.
This shows that there exists a one parameter family of cubics
passing through \(x_{1},\ldots,x_{9}\) and therefore 
implies the existence of the cubic \(\underline{D}\) other than \(xy(x+y)\).



 \subsubsection{{\bf Type $\mathrm{IV}$}}
 \label{subsubsec:type-iv-construction}
 A type IV fibre consists of two smooth rational curves tangent at a point. In this case, $\underline{D}$ is union of a line $L$ and a conic $Q$ tangent at $p$. Then $Y$ is blow up of six points on 
 $C\setminus \{p\}$ and three points on $L\setminus \{p\}$. 

 To construct a reference marked ALG pair of type
 \(\mathrm{IV}\), we simply fix a smooth cubic \(C\) which
 intersects \(\underline{D}\setminus \{p\}\) at 9 distinct points.
 Denote by \(p_{\mathfrak{r},1},p_{\mathfrak{r},2},p_{\mathfrak{r},3}\) 
 the intersection \(L\cap C\)
 and \(p_{\mathfrak{r},4},\ldots,p_{\mathfrak{r},9}\in Q\cap C\).
 Consider the blow-up \(Y_{\mathfrak{r}}=\mathrm{Bl}_{p_{
 \mathfrak{r},1},\ldots,p_{\mathfrak{r},9}}\mathbf{P}^{2}\)
 and \(D_{\mathfrak{r}}\), the proper transform of \(\underline{D}\). 
 Let \(E_{\mathfrak{r},i}\)
 be the exceptional divisor over \(p_{\mathfrak{r},i}\).
 In which case, we can choose
 \begin{equation*}
 \mathcal{B}_{\mathfrak{r}}:=\{H-E_{\mathfrak{r},4}-E_{\mathfrak{r},7}-E_{\mathfrak{r},1}\}
 \cup\bigcup_{i=2,3}\{E_{\mathfrak{r},1}-E_{\mathfrak{r},i}\}\cup\bigcup_{i=5,6}
 \{E_{\mathfrak{r},4}-E_{\mathfrak{r},i}\}\cup\bigcup_{i=8,9}
 \{E_{\mathfrak{r},7}-E_{\mathfrak{r},i}\}
 \end{equation*}
 to be our basis of \(\mathrm{H}_{2}(X_{\mathfrak{r}},\mathbb{Z})\).
  For simplicity, we denote the elements in \(\mathcal{B}_{\mathfrak{r}}\)
  by \(\gamma_{\mathfrak{r},1},\ldots,\gamma_{\mathfrak{r},8}\).
  The fibre class is represented by 
  \begin{equation*}
  3\gamma_{\mathfrak{r},1} + \sum_{j=2}^{8} \gamma_{\mathfrak{r},j}.
  \end{equation*}
  We may assume \(\underline{D}=\{y(x^{2}+yz)=0\}\);
  we can achieve this using the
  \(\mathrm{PGL}(3,\mathbb{C})\) action. Let
  \begin{equation}
  \label{eq:meromorphic-form-iv}
  \Omega = \frac{x\mathrm{d}y\wedge\mathrm{d}z-
  y\mathrm{d}x\wedge\mathrm{d}z+z\mathrm{d}x\wedge\mathrm{d}y}{y(x^{2}+yz)}.
  \end{equation}
  It follows that 
  \begin{equation}
  \operatorname{Im}(\mathcal{P})\subset 
  \left\{(y_{1},\ldots,y_{8})~\Big|~3y_{1}+\sum_{j=2}^{8} y_{j}=0\right\}.
  \end{equation}
  Under the affine coordinates \(u=x/z\) and \(v=y/z\), we have
  \begin{equation}
  \label{eq:meromorphic-form-iv-chart}
  \Omega = \frac{\mathrm{d}u\wedge\mathrm{d}v}{v(u^{2}+v)}.
  \end{equation}
  One can easily check that
  \begin{align*}
  \operatorname{Res}_{L}\Omega = \frac{\mathrm{d} u}{u^{2}},~\mbox{and}~
  \operatorname{Res}_{Q}\Omega = \frac{\mathrm{d} u}{u^{2}}.
  \end{align*}
  We need the following computational results.
  \begin{lem}
  \label{lem:period-type-iv}
  Let \(\underline{D}=\{y(x^{2}+yz)=0\}\subset\mathbf{P}^{2}\)
  and \(p=[0\mathpunct{:}0\mathpunct{:}1]\)
  be the unique singular point on \(\underline{D}\). Let
  \(\Omega\) be the meromorphic two form defined in \eqref{eq:meromorphic-form-iv}.
  Let \(x_{1},\ldots,x_{9}\in \underline{D}\setminus \{p\}\) 
  such that \(x_{1},x_{2},x_{3}\in L\) and \(x_{4},\ldots,x_{9}\in Q\). 
  Let \(Y=\mathrm{Bl}_{\{x_{1},\ldots,x_{9}\}}\mathbf{P}^{2}\).
  Denote by \(E_{i}\) the exceptional divisor over \(x_{i}\).
  Let \(x_{i}=[a_{i}\mathpunct{:}0\mathpunct{:}c_{i}]\) for \(i=1,2,3\)
  and \(x_{i}=[a_{i}\mathpunct{:}b_{i}\mathpunct{:}c_{i}]\) for \(i=4,\ldots,9\).
  Then
  \begin{itemize}
  \item[(a)] Assume that \(x_{4}\ne x_{7}\).
  Then the line \(\overline{x_{4}x_{7}}\) intersects \(y=0\)
  at \([c_{7}b_{4}-c_{4}b_{7}\mathpunct{:}0
  \mathpunct{:}a_{4}b_{7}-a_{7}b_{4}]\ne p\).
  Let \(H\) be the hyperplane class in \(\mathbf{P}^{2}\). 
  Then we have
  \begin{equation*}
  \int_{\gamma_{1}}\Omega =
  \int_{H-E_{1}-E_{4}-E_{7}}\Omega = \frac{a_{4}b_{7}-a_{7}b_{4}}{c_{7}b_{4}-c_{4}b_{7}}-
  \frac{c_{1}}{a_{1}}.
  \end{equation*}
  \item[(b)] We have for \(k=1,4,7\)
  \begin{equation*}
  \int_{E_{k}-E_{k+j}}\Omega = \frac{c_{k+j}}{a_{k+j}}-\frac{c_{k}}{a_{k}}~\mbox{for}~j=1,2.
  \end{equation*}
  \end{itemize}
  \end{lem}
  \begin{proof}
  This follows from the formulae 
  \begin{align*}
  \operatorname{Res}_{L}\Omega = \frac{\mathrm{d} u}{u^{2}},~\mbox{and}~
  \operatorname{Res}_{Q}\Omega = \frac{\mathrm{d} u}{u^{2}}
  \end{align*}
  and the residue theorem. The proof is hence omitted.
  \end{proof}
  By Lemma \ref{lem:period-type-iv}, the vector \((y_{1},\ldots,y_{8})\in\mathbb{C}^{8}\)
determines \(x_{1},\ldots,x_{9}\) on \(\underline{D}\setminus\{p\}\).
Indeed, we can put \(x_{4}=p_{\mathfrak{r},4}\) and 
\(x_{7}=p_{\mathfrak{r},7}\) and the results in (a) and (b) in
Lemma \ref{lem:period-type-iv} would determine the location of all the rest \(x_{i}\)'s.
The only thing we have to show is that 
\(Y=\mathrm{Bl}_{\{x_{1},\ldots,x_{9}\}}\mathbf{P}^{2}\)
is a rational elliptic surface, i.e.,
there is a smooth cubic passing through \(x_{1},\ldots,x_{9}\).

Again it suffices to construct a cubic passing through
the points \(x_{1},\ldots,x_{9}\). This can be done directly. Indeed, 
suppose the coordinate of \(x_{i}\) is given as
in Lemma \ref{lem:period-type-iv}. The cubic defined by 
\begin{equation}
\prod_{i=1,2,3} (c_{i}x-a_{i}z) + y\cdot (ax^{2} + by^2 + cz^{2} + dxy + eyz + fxz)=0
\end{equation}
passing through \(x_{1},x_{2},x_{3}\). 
Note that the rational curve \(Q\) is parameterized by
\begin{equation*}
[\alpha\mathpunct{:}\beta]\mapsto [\alpha\beta\mathpunct{:}\alpha^{2}\mathpunct{:}\beta^{2}].
\end{equation*}
Now set \(yz=-x^{2}\) in the above equation. It
follows that \(b\), \(c\), \(d\), \(f\) and \(a-e\) are uniquely determined.
This shows that there exists a one parameter family of cubics
passing through \(x_{1},\ldots,x_{9}\) and therefore 
implies the existence of the cubic \(\underline{D}\) other than \(y(x^{2}+yz)\).

 \subsubsection{{\bf Type $\mathrm{II}^{\ast}$}} 
 \label{subsubsec:type-ii*-construction}
 A type \(\mathrm{II}^{\ast}\)
 fibre is the $E_8$ configuration. 
 Assume that $Y$ is a rational elliptic surface with an type $\mathrm{II}^{\ast}$ fibre $D$. The section of $Y$ can only intersect the unique component of $D$ with multiplicity one. One can then iteratively contracts the section, the component with multiplicity $1,2,3,4,5,6,4,2$ (in total nine curves) and end up with a smooth projective surface of Picard number one, that is, $\mathbf{P}^2$. The only non-contracted component of $D$ in the process has multiplicity three.    
 In other words, any rational elliptic surface with a type $\mathrm{II}^{\ast}$ fibre can be realized as blow up on the base points of the cubic pencil containing a triple line which is tri-tangent to any other smooth element in the pencil. If the pencil contains a cusp curve, then the singular configuration of $Y$ is $\mathrm{II}^{\ast}\mathrm{II}$. Otherwise, the pencil contains a nodal curve and the singular configuration of $Y$ is $\mathrm{II}^{\ast}\mathrm{I}_1^2$. From the long exact sequence \eqref{eq:long-exact-sequence}, we have $\mathrm{H}_2(X)$ is of rank one and generated by the fibre class of $Y$. Thus, the periods all vanish in both cases. Depending on the pencil contains a cusp cubic or not, there are exactly two different rational elliptic surfaces with an $\mathrm{II}^{\ast}$ fibre. Both of them are extremal rational elliptic surfaces (rational elliptic surfaces whose relative automorphism group is finite, cf.~\cite{MP}) and their singular configuration is $\mathrm{II}^{\ast}\mathrm{II}$ or $\mathrm{II}^{\ast}\mathrm{I}_{1}^{2}$. One can have an isotrivial deformation of the latter which degenerates to the former. 
In particular, the periods won't distinguish these two cases. 


 \subsubsection{{\bf Type $\mathrm{III}^{\ast}$}}
 \label{subsubsec:type-iii*-construction}
 A type \(\mathrm{III}^{\ast}\) fibre is the $E_7$ configuration.
 Assume that $Y$ is a rational elliptic surface with a type $\mathrm{III}^{\ast}$ fibre $D$. Sections of $Y$ must intersect a component of $D$ with multiplicity one. One can then iteratively contracts the section, followed by the components with multiplicity $1,2,3,4,2$ (six curves in total) and end up with a smooth projective surface $Y'$ of Picard group rank four. Denote by $C_1,C_2,C_3$ the image of the remaining components of $D$ with multiplicity $1,2,3$ respectively. Recall that $\mathbf{P}^1\times \mathbf{P}^1$ contains no curve with negative self-intersection and $\mathbb{F}_2$ contains a unique one such curve. We can conclude that the minimal model of $Y'$ must be $\mathbf{P}^2$ since $C_1^2=C_2^2=-2$. Then the $(-1)$-curve must intersect $C_1$ otherwise the image of $C_3$ to $\mathbf{P}^2$ would have self-intersection $2$ which is absurd. One may iteratively contracts the sections and the components with multiplicity $1,2,3,4,3,2$ (seven curves in total) from $Y$. Denote the resulting smooth projective surface by $Y''$ and the image of the remaining components of $D$ with multiplicity $i$ by $D_i\subseteq Y''$ for $i=1,2$. From the earlier discussion $D_1$ must intersect a $(-1)$-curve. We claim that $Y''$ then becomes the Hirzebruch surface $\mathbb{F}_1$. Indeed, any irreducible curve in $Y$ has self-intersection at least $-2$, so we may exclude the possibility of $\mathbb{F}_n$, $n\geq 3$. Notice that $D_2^2=1$ and the self-intersection pairing in $\mathbf{P}^1\times \mathbf{P}^1$ or $\mathbb{F}_2$ are even. So the claim is established. From $D_1^2=0$ and Riemann--Roch theorem, $D_1$ must be a fibre. In particular, $D_1$ intersects a $(-1)$-curve. Therefore, after contracting this $(-1)$-curve to $\mathbf{P}^2$, the image of $D$ is a union of a double line and a line. 
 
 To sum up, any rational elliptic surface with an $\mathrm{III}^{\ast}$ fibre can be realized as a blow-up of the base locus of the pencil spanned by a smooth cubic \(\underline{D}\) and a union of a double line $M$ and a line $N$, with the double line intersecting the smooth cubic at its flex point and the other line \(N\) also passes through the flex point of the cubic. 
 

There is another way to construct a rational elliptic
surface with an \(\mathrm{III}^{\ast}\)-fibre which
is easier to calculate the periods. 
Consider a triple line \(\underline{D}=3L\) in \(\mathbf{P}^{2}\).
Take \(C\) to be a smooth cubic which is tangent at \(p\in \underline{D}\)
and intersects transversally at another point \(q\in \underline{D}\).
Blowing up \(p\) and \(q\) yields a rational elliptic surface
with an \(\mathrm{III}^{\ast}\) fibre.
Explicitly, if we denote by \([x\mathpunct{:}y\mathpunct{:}z]\)
the coordinate on \(\mathbf{P}^{2}\),
we can take \(\underline{D}_{\mathrm{red}}
=\{x=0\}\) and \(C\)
to be the plane curve defined by 
\begin{equation*}
y^{2}z + x(z^{2}+xy+a'yz),~a'\in\mathbb{C}^{\ast}.
\end{equation*}
In which case, \(p=[0\mathpunct{:}0\mathpunct{:}1]\)
and \(q=[0\mathpunct{:}1\mathpunct{:}0]\)
One checks that this is smooth 
whenever \(q^{3}\ne -27\). 
Using change of variables, the equation displayed above
can be transformed into 
\begin{equation}
\label{eq:type-iii-*}
y^{2}z + x(z^{2}+axy+yz),~a\in\mathbb{C}^{\ast}.
\end{equation}
We see that \eqref{eq:type-iii-*}
is smooth for general \(a\).
If it happens that \eqref{eq:type-iii-*} is singular, we can 
always add a multiple of \(x^{3}\)
to the equation to make it smooth. In any case,
we obtain a rational elliptic surface with 
singular fibre configuration \(\mathrm{III}^{\ast}\) at infinity.

To obtain cycles in \(\mathrm{H}_{2}(X,\mathbb{Z})\),
let \(T'\) be the tangent line of \(C\) at 
\([0\mathpunct{:}1\mathpunct{:}0]\), i.e., \(T'=\{z=0\}\).
After blowing-ups, the proper transform \(T\) of \(T'\) becomes a 
\((-1)\) curve and therefore it is a section.
Then \(\gamma_{1}:=[T]-[E]\), where \(E\) is the section obtained in the last step of
blow-ups of \(\mathbf{P}^{2}\) at \([0\mathpunct{:}1\mathpunct{:}0]\),
gives an element in \(\mathrm{H}_{2}(X,\mathbb{Z})\).
The fibre class \([f]\) 
gives another element in \(\mathrm{H}_{2}(X,\mathbb{Z})\).
One can check \(\mathcal{B}:=\{\gamma_{1},[f]\}\) is a basis
of \(\mathrm{H}_{2}(X,\mathbb{Z})\).
As before, we shall pick a smooth cubic and a basis of the 
homology of its complement (after blow-ups) as above
to serve our marked ALG pair of type \(\mathrm{III}^{\ast}\).
We denote the pair by \((Y_{\mathfrak{r}},D_{\mathfrak{r}})\)
and the basis by \(\mathcal{B}_{\mathfrak{r}}\).


 We now
 choose a section of \(\Omega_{\mathbf{P}^{2}}(\underline{D})\)
 \begin{equation}
 \Omega:=\frac{x\mathrm{d}y\wedge\mathrm{d}z-y\mathrm{d}x\wedge
 \mathrm{d}z+z\mathrm{d}x\wedge\mathrm{d}y}{x^{3}}.
 \end{equation}
 We shall compute the periods using the two form \(\Omega\).

Let us investigate the blow-ups over \([0\mathpunct{:}1\mathpunct{:}0]\) first.
Using the affine coordinates \(u:=x/y\) and \(v:=z/y\),
the form \(\Omega\) is transformed into
\begin{equation}
\Omega = -\frac{\mathrm{d}u\wedge\mathrm{d}v}{u^{3}}
\end{equation}
and \(C\) is defined by 
\begin{equation}
\{v + u(v^{2}+au+v)=0\}.
\end{equation}
Now we compute the blow-up. Set \(v=us\) (here \(s\) is 
the coordinate on \(\mathbf{P}^{1}\)). We then have
\begin{equation}
\Omega = -\frac{\mathrm{d}u\wedge\mathrm{d}s}{u^{2}}.
\end{equation}
Here \(\{u=0\}\) corresponds to the expectional
divisor (with multiplicity two as expected). In the meanwhile,
the proper transform of \(C\) is 
\begin{equation}
\{s+u^{2}s^{2}+au + us=0\}
\end{equation}
and the proper transform of \(T'\) is \(\{s=0\}\). 
We blow up at \((u,s)=(0,0)\) one more time.
Let \(s=ut\). Then the meromorphic two form
becomes
\begin{equation*}
\Omega = -\frac{\mathrm{d}u\wedge\mathrm{d}t}{u}.
\end{equation*}
The proper transform of \(C\) is 
\begin{equation}
\{t + u^{3}t^{2} + a + ut=0\}
\end{equation}
and the proper transform of \(T'\) is defined by
\(\{t=0\}\). Denote by \(E'\) the exceptional
divisor of the second blow-up. By our convention, \(t\) serves as an affine
coordinate on \(E'\cong \mathbf{P}^{1}\).
In order to achieve \(Y\), we need one more blow up at \((u,t)=(0,-a)\),
the intersection of the proper transform of \(C\) and \(E'\).
Denote by \(E\) the exceptional divisor and by \(T\)
the proper transform of \(T'\). Then 
\(\gamma_{1}:=[E]-[T]\) represents a homology cycle in \(X:=Y\setminus D\).
The cycle \(\gamma_{1}\) together with the fibre class \([f]\) 
form a basis \(\mathcal{B}\) of \(\mathrm{H}_{2}(X,\mathbb{Z})\).
One can compute
\begin{equation}
\int_{[E]-[T]} \Omega = \int_{0}^{-a}\mathrm{d}t = -a.
\end{equation}
By varying \(a\), we have proven that
\begin{equation}
\operatorname{Im}(\mathcal{P})=\{(y_{1},y_{2})~|~y_{2}=0\}
\end{equation}
where \(y_{1}\) (resp.~\(y_{2}\)) is the coordinate corresponding to 
\begin{equation}
\int_{\gamma_{1}}\Omega,~\quad\mbox{(resp.~\(\displaystyle\int_{[f]}\Omega \equiv 0\).)}
\end{equation}

Let us now describe the moduli space of rational elliptic surfaces with a \(\mathrm{III}^{\ast}\)-fibre.
From the classification of the singular configuration of a rational elliptic surface $Y$ containing a type $\mathrm{III}^{\ast}$-fibre, $Y$ must contain an $\mathrm{I}_1$-fibre unless its singular configuration is $\mathrm{III}^{\ast}\mathrm{III}$. 
We have the following two cases:
\begin{itemize}
\item[(a)] The pencil contains a nodal curve \(C\).
Up to the $\mathrm{PGL}(3)$-action on $\mathbf{P}^2$, we may assume that the nodal curve $C$ is of the form $x^{3}+y^{3}+xyz=0$ with a node at \(p=[0\mathpunct{:}0\mathpunct{:}1]\). 
Indeed, if \(C\) is a nodal curve with a node at \(p\), we can always move \(p\) to \([0\mathpunct{:}0\mathpunct{:}1]\).
Let \(F(x,y,z)\) be the defining equation of \(C\) and 
\(f(x,y)=F(x,y,1)\) be the equation of \(C\) on the affine chart \(\{z\ne 0\}\).
We may further use the $\mathrm{PGL}(3)$-action to assume that 
\begin{equation}
f(x,y) = xy + g(x,y)
\end{equation}
where \(g(x,y)\) is homogeneous of degree \(3\). In other words, \(F(x,y,z)=xyz+g(x,y)\).
Now we can use the remaining symmetries to eliminate the \(x^{2}y\) and \(xy^{2}\) terms
in \(g\) as well as adjust the coefficients of \(x^{3}\) and \(y^{3}\).
As a result, we achieve the equation \(x^{3}+y^{3}+xyz=0\).
It is known that there is an isomorphism 
\begin{equation*}
\mathbb{C}^{\ast}\cong C\setminus \{p\},~t\mapsto [t\mathpunct{:}-t^{2}\mathpunct{:}1-t^{3}].
\end{equation*} 
The flex points are located at $[1\mathpunct{:}-1\mathpunct{:}0]$, $[\omega\mathpunct{:}-\omega^{2}\mathpunct{:}0]$, and $[\omega^{2}\mathpunct{:}-\omega\mathpunct{:}0]$ where \(\omega\) is the primitive 3\textsuperscript{rd} root of unity. Moreover, these three flex points are equivalent under the $\mathrm{PGL}(3)$-action. One can easily check that if \(A\in \mathrm{PGL}(3,\mathbb{C})\) leaves \(x^{3}+y^{3}+xyz\) invariant and fixes $[1\mathpunct{:}-1\mathpunct{:}0]$, then either \(A=\mathrm{id}\) or \(A\colon x\mapsto y,~y\mapsto x,~z\mapsto z\). Then $M$ is the tangent line of $C$ at $[1\mathpunct{:}-1\mathpunct{:}0]$ and $N$ can be any line passing through $[1\mathpunct{:}-1\mathpunct{:}0]$. In particular, by rotating \(N\), we obtain a $\mathbf{P}^1$-family of rational elliptic surfaces with a $\mathrm{III}^{\ast}$-fibre. If $N$ meets the node of $C$, then the rational elliptic surface contains an $\mathrm{I}_2$-fibre. If $N$ is tangent to a smooth point of $C$, then the rational elliptic surface contains an $\mathrm{II}$-fibre. 
To sum up, the moduli space of rational elliptic surfaces with singular fibes of \(\mathrm{III}^{\ast}\) and \(\mathrm{I}_{1}\) is \(\mathbb{C}^{\ast}\subset\mathbf{P}^{1}\). The boundary points parameterize the rational elliptic surfaces with singular configuration \(\mathrm{II}^{\ast}\mathrm{I}_{1}^{2}\) and \(\mathrm{III}^{\ast}\mathrm{I}_{2}\mathrm{I}_{1}\).
\item[(b)] The pencil contains a cuspidal curve \(C\). One can use the $\mathrm{PGL}(3)$-action to assume that $C=\{y^2z=x^3\}$ and $p=[0\mathpunct{:}0\mathpunct{:}1]$ is the cusp. It is known that \(C\) the complement \(C\setminus \{p\}\) is isomorphic to \(\mathbb{C}\) as an
additive group via 
\begin{equation*}
\mathbb{C}\cong C\setminus \{p\},~t\mapsto [t\mathpunct{:}1\mathpunct{:}t^{3}]
\end{equation*}
and \(C\setminus\{p\}\) admits a unique flex point. 
(Recall that the group law on \(C\setminus\{p\}\) is defined
in the same manner as the one defined on elliptic curves. For \(P\) and \(Q\) on \(C\setminus \{p\}\), \(P+Q\) is the point \(R\in C\setminus\{p\}\) such that \(P\), \(Q\), and \(R\) are colinear.) Let \(M\) be the tangent of \(C\) at the flex point and \(N\) be a line passing through it. As in the previous case, rotating \(N\) gives rise to a \(\mathbf{P}^{1}\)-family of rational elliptic surfaces. If $N$ passes through smooth points of $C$, we obtain a rational elliptic surface with singular configuration $\mathrm{III}^{\ast}\mathrm{II}~\mathrm{I}_1$. Moreover, any such two lines determine the same rational elliptic surface. When $N$ passes through the cusp of $C$, then the resulting rational elliptic surface has the singular configuration $\mathrm{III}^{\ast} \mathrm{III}$. When \(N\) is also tangent to \(C\), the corresponding rational elliptic surface has the singular configuration \(\mathrm{II}^{\ast}\mathrm{II}\). As a summary, the parameter space of rational elliptic surface with a type \(\mathrm{III}^{\ast}\) fibre is a $\mathbf{P}^1$. The generic point of $\mathbf{P}^1$ parametrizes those with singular configuration $\mathrm{III}^{\ast}\mathrm{I}_1^{3}$ which admits degenerations to $\mathrm{III}^{\ast}\mathrm{II}~\mathrm{I}_1$ and to $\mathrm{II}^{\ast}\mathrm{II}$. 
\end{itemize}
In particular, there are three rational elliptic surfaces with trivial periods, with singular configuration $\mathrm{III}^{\ast}\mathrm{III}$, $\mathrm{III}^{\ast}\mathrm{II}~\mathrm{I}_1$, and $\mathrm{III}^{\ast}\mathrm{I}_2\mathrm{I}_1$.  
 
%
%

 \subsubsection{{\bf Type $\mathrm{IV}^{\ast}$}}
 \label{subsubsec:type-iv*-construction}
 A type \(\mathrm{IV}^{\ast}\) is the $E_6$ configuration.

 To construct the model, we consider \(\underline{D}=\{x^{3}=0\}\).
 Let \(C\) be a cubic of the form
 \begin{equation*}
 yz(y-z)+x(cxy+xz+dyz)~\mbox{with}~c,d\in\mathbb{C}.
 \end{equation*}
 Then \(C\) intersects \(\underline{D}_{\mathrm{red}}\) at \([0\mathpunct{:}0\mathpunct{:}1]\),
 \([0\mathpunct{:}1\mathpunct{:}1]\),
 \([0\mathpunct{:}1\mathpunct{:}0]\) with all multiplicity one.
 Moreover, one can check that for any \(c,d\in\mathbb{C}\),
 the linear system spanned by \(C\) and \(\underline{D}\)
 contains a smooth member. As before, we 
 pick a smooth cubic \(C\) (with constants 
 \(c_{\mathfrak{r}}\) and \(d_{\mathfrak{r}}\) in
 the equation) to build our 
 marked reference ALG pair of type \(\mathrm{IV}^{\ast}\).

 Recall that in the present situation, one can achieve
 a desired pair \((Y_{\mathfrak{r}},D_{\mathfrak{r}})\) by blowing up at 
 those intersections \(C\cap \underline{D}\) (we 
 blow up three times at each intersection point and
 there are nine blow-ups needed in total).
 Denote by \(E_{\mathfrak{r},0}\) and \(E_{\mathfrak{r},\infty}\) the exceptional
 divisor over \([0\mathpunct{:}0\mathpunct{:}1]\)
 and \([0\mathpunct{:}1\mathpunct{:}0]\) from the last (the third) blow-up.
 Let \(T_{\mathfrak{r},0}'\) (resp.~\(T_{\mathfrak{r},\infty}'\)) be the tangent line of \(C\) 
 at \([0\mathpunct{:}0\mathpunct{:}1]\)
 (resp.~\([0\mathpunct{:}1\mathpunct{:}0]\)) and
 \(T_{\mathfrak{r},0}\) (resp.~\(T_{\mathfrak{r},\infty}\)) be the proper 
 transform on \(Y\). It is easy to check that 
 \begin{equation}
 \mathcal{B}_{\mathfrak{r}}:=\{
 [E_{\mathfrak{r},0}]-[T_{\mathfrak{r},0}],~[E_{\mathfrak{r},\infty}]-
 [T_{\mathfrak{r},\infty}],[f]\}
 \end{equation}
 is a basis of \(\mathrm{H}_{2}(X_{\mathfrak{r}},\mathbb{Z})\) 
 where \(X_{\mathfrak{r}}=Y_{\mathfrak{r}}\setminus D_{\mathfrak{r}}\) as before.

 Like in the previous case, we take
 \begin{equation}
 \Omega = \frac{x\mathrm{d}y\wedge\mathrm{d}z-
  y\mathrm{d}x\wedge\mathrm{d}z+z\mathrm{d}x\wedge\mathrm{d}y}{x^{3}}.
 \end{equation}
 For arbitrary \(c,d\in\mathbb{C}\), we can compute (as in the case of type 
 \(\mathrm{III}^{\ast}\))
 \begin{equation*}
 \int_{[E_{0}]-[T_{0}]}\Omega = -c,~\mbox{and}~
 \int_{[E_{\infty}]-[T_{\infty}]}\Omega = -d.
 \end{equation*}
 This shows that
 \begin{equation}
\operatorname{Im}(\mathcal{P})=\{(y_{1},y_{2},y_{3})~|~y_{3}=0\}.
\end{equation}

\begin{figure}
\includegraphics[scale=0.8]{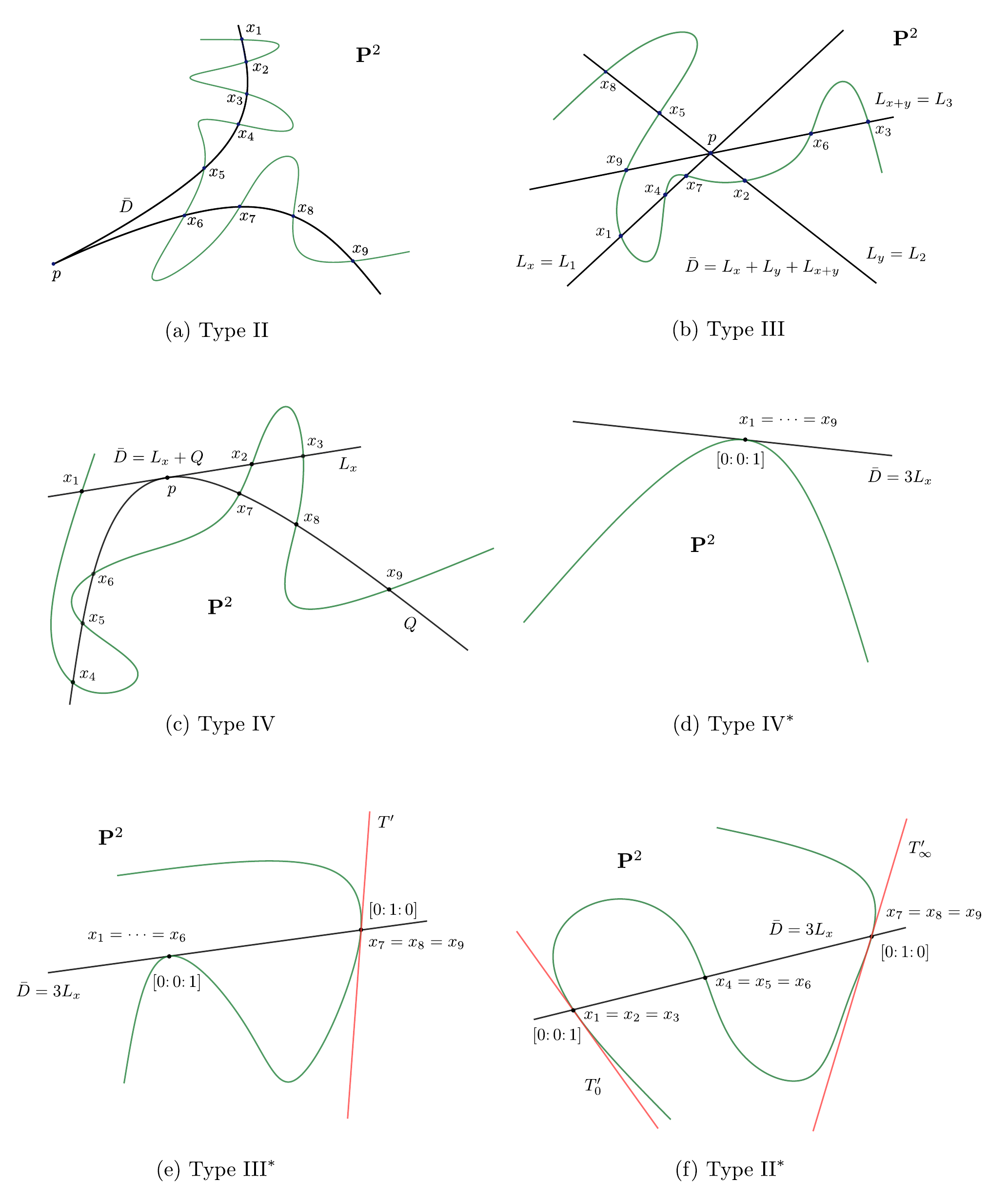}
\caption{The constructions of 
\(\mathrm{ALG}\) pairs of type \(\mathrm{II}\), \(\mathrm{III}\),
\(\mathrm{IV}\), \(\mathrm{IV}^{\ast}\), \(\mathrm{III}^{\ast}\),
and \(\mathrm{II}^{\ast}\). The black lines stand for the singular divisor \(\underline{D}\),
the green curves stand for the cubic \(C\) in the corresponding pencil, and the
red lines are the image of the \((-1)\) curves under \(Y\to\mathbf{P}^{2}\).}
\label{fig:alg-pair}
\end{figure}

\subsubsection{{\bf Type $\mathrm{I}_{0}^{\ast}$}}
\label{subsubsec:type-i0*-construction}
A type \(\mathrm{I}_{0}^{\ast}\) fibre is the \(D_{4}\) configuration.
We begin with \(\underline{D}=\{x^{2}y=0\}\) and consider a cubic \(C\) intersecting
with \(\underline{D}\) at three distinct points on each irreducible component.
Using the \(\mathrm{PGL}(3,\mathbb{C})\)-action on \(\mathbf{P}^{2}\), 
we may assume that 
\begin{itemize}
\item \(C\cap \{x=0\}=
\left\{[0\mathpunct{:}1\mathpunct{:}0], 
[0\mathpunct{:}1\mathpunct{:}1],
[0\mathpunct{:}1\mathpunct{:}a]\right\}\);
\item \(C\cap \{y=0\}=
\left\{[1\mathpunct{:}0\mathpunct{:}0], 
[1\mathpunct{:}0\mathpunct{:}b],
[1\mathpunct{:}0\mathpunct{:}c]\right\}\) with \(b,c\in\mathbb{C}^{\ast}\) distinct;
\end{itemize} 
It follows that \(C\) is 
defined by the following equation
\begin{equation*}
z(y-z)(ay-z)+x\left(dy^{2}-(b+c)z^{2}+fxy+bcxz+hyz\right)=0.
\end{equation*}
We may as well assume that \(f=0\) by adding the
defining equation of \(\underline{D}\). To summarize,
we may assume that \(C\) is given by 
\begin{equation*}
z(y-z)(ay-z)+x\left(dy^{2}-(b+c)z^{2}+bcxz+hyz\right)=0.
\end{equation*}
Picking a smooth cubic \(C\) of the form above and
blowing-up at those points yield a 
ALG pair \((Y,D)\) of type \(\mathrm{I}_{0}^{\ast}\). 
We need to construct a basis of \(\mathrm{H}_{2}(X,\mathbb{Z})\),
i.e., we need to construct five cycles lying in \([D]^{\perp}\)
under the identification.
\begin{itemize}
\item Consider the tangent of \(C\) at \([0\mathpunct{:}1\mathpunct{:}0]\).
Explicitly, it is given by  
\begin{equation*}
dx+az=0
\end{equation*}
which intersects \(\{y=0\}\) at 
\([a\mathpunct{:}0\mathpunct{:}-d]\).
Denote by \(L_{1}\) its proper transform on \(Y\);
\item Consider the tangent of \(C\) at \([0\mathpunct{:}1\mathpunct{:}1]\).
Explicitly, it is given by  
\begin{equation*}
\left(d+h-b-c\right)x-(1-a)y+(1-a)z=0.
\end{equation*}
It intersects \(\{y=0\}\) at 
\begin{equation*}
\left[1-a\mathpunct{:}0\mathpunct{:}b+c-d-h\right].
\end{equation*}
Denote by \(L_{2}\) its proper transform on \(Y\);
\item Consider the tangent of \(C\) at \([0\mathpunct{:}1\mathpunct{:}a]\).
Explicitly, it is given by  
\begin{equation*}
(-a^{2}(b+c)+d+ah)x-a(1-a)(z-ay)=0
\end{equation*}
It intersects \(\{y=0\}\) at 
\begin{equation*}
\left[a(1-a)\mathpunct{:}0\mathpunct{:}-a^{2}(b+c)+d+ah)\right].
\end{equation*}
Denote by \(L_{2}\) its proper transform on \(Y\);

\item Let \(E_{7}\) be the exceptional divisor over 
\([1\mathpunct{:}0\mathpunct{:}b]\) in the second blow-up;
\item Let \(E_{8}\) be the exceptional divisor over 
\([1\mathpunct{:}0\mathpunct{:}c]\);
\item Let \(E_{9}\) be the exceptional divisor over \([1\mathpunct{:}0\mathpunct{:}0]\).
\end{itemize}
Now we can construct our cycles via
\begin{itemize}
\item[(1)] \(\gamma_{1}=L_{1}-E_{9}\);
\item[(2)] \(\gamma_{2}=L_{2}-E_{9}\);
\item[(3)] \(\gamma_{3}=L_{3}-E_{9}\);
\item[(4)] \(\gamma_{4}=E_{9}-E_{7}\);
\item[(5)] \(\gamma_{5}=E_{9}-E_{8}\).
\end{itemize}
One can easily check that each of them lies in \([D]^{\perp}\)
and they form a
basis of \(\mathrm{H}_{2}(X,\mathbb{Z})\). 
Also the fibre class \([f]\) is
given by \(\sum_{i=1}^{5}[\gamma_{i}]\). In the 
present case, we take
\begin{equation*}
\Omega = \frac{x\mathrm{d}y\wedge\mathrm{d}z-
  y\mathrm{d}x\wedge\mathrm{d}z+z\mathrm{d}x\wedge\mathrm{d}y}{x^{2}y}.
\end{equation*}
It is also straightforward to check (parallel
to the computation in previous sections) that the set of period vectors
\begin{equation*}
\left\{(y_{1},\ldots,y_{5})~\Big|~y_{i}:=\int_{[\gamma_{i}]}\Omega,~i=1,\ldots,5\right\}
\end{equation*}
is equal to \(\{(y_{1},\ldots,y_{5})~|~\sum_{i=1}^{5}y_{i}=0\}\).
Indeed, one can check that
\begin{equation*}
\int_{\gamma_{1}}\Omega = \frac{-d}{a},~
\int_{\gamma_{2}}\Omega = \frac{b+c-d-h}{1-a},~
\int_{\gamma_{3}}\Omega = \frac{-a^{2}(b+c)+d+ah}{a(1-a)},~
\int_{\gamma_{4}}\Omega = -b,~\mbox{and}~
\int_{\gamma_{5}}\Omega = -c.
\end{equation*}
This proves the surjectivity of the period
map in this case.

\subsubsection{{\bf Type $\mathrm{I}_{1}^{\ast}$}}
\label{subsubsec:type-i1*-construction}

A type \(\mathrm{I}_{1}^{\ast}\) fibre is the \(D_{6}\) configuration.
We begin with \(\underline{D}=\{x^{2}y=0\}\) and consider a cubic \(C\) intersecting
with \(\underline{D}\) at three distinct points on \(\{y=0\}\) but 
tangent to \(\{x=0\}\).
Using the \(\mathrm{PGL}(3,\mathbb{C})\)-action on \(\mathbf{P}^{2}\), 
we may assume that 
\begin{itemize}
\item \(C\cap \{x=0\}=
\left\{[0\mathpunct{:}1\mathpunct{:}0], 
[0\mathpunct{:}1\mathpunct{:}a]\right\}\)
and \(C\) is tangent to \(\{x=0\}\) at \([0\mathpunct{:}1\mathpunct{:}0]\).
\item \(C\cap \{y=0\}=
\left\{[1\mathpunct{:}0\mathpunct{:}0], 
[1\mathpunct{:}0\mathpunct{:}b],
[1\mathpunct{:}0\mathpunct{:}c]\right\}\) with \(b\ne c\);
\end{itemize} 
It follows that \(C\) is 
defined by the following equation
\begin{equation*}
z^{2}(ay-z)+x\left(dy^{2}+(b+c)z^{2}+fxy-bcxz+hyz\right)=0.
\end{equation*}
We may as well assume that \(f=0\) by adding the
defining equation of \(\underline{D}\). To summarize,
we may assume that \(C\) is given by 
\begin{equation*}
z^{2}(ay-z)+x\left(dy^{2}+(b+c)z^{2}-bcxz+hyz\right)=0.
\end{equation*}
Picking a smooth cubic \(C\) of the form above and
blowing-up at those points yield an
ALG pair \((Y,D)\) of type \(\mathrm{I}_{1}^{\ast}\). 
Note that \(C\) is smooth implies \(d\ne 0\).
We need to construct a basis of \(\mathrm{H}_{2}(X,\mathbb{Z})\),
i.e., we need to construct four cycles lying in \([D]^{\perp}\)
under the identification.
\begin{itemize}
\item Consider the tangent of \(C\) at 
\([0\mathpunct{:}1\mathpunct{:}a]\).
Explicitly, it is given by  
\begin{equation*}
(a^{2}(b+c)+d+ah)x-a^{2}(z-ay)=0.
\end{equation*}
which intersects \(\{y=0\}\) at 
\begin{equation*}
\left[a^{2}\mathpunct{:}0\mathpunct{:}a^{2}(b+c)+d+ah\right].
\end{equation*}
Denote by \(L\) its proper transform on \(Y\);
\item Consider a conic passing through \([1\mathpunct{:}0\mathpunct{:}0]\)
and tangent to \(\{x=0\}\)
at \([0\mathpunct{:}1\mathpunct{:}0]\) such that the intersection at
\([0\mathpunct{:}1\mathpunct{:}0]\) with \(C\) has multiplicity \(3\).
Explicitly, when \(d\ne 0\),
we could take for example
\begin{equation*}
az^{2}+dxy-abxz=0.
\end{equation*}
It intersects \(\{y=0\}\) at 
\begin{equation*}
\left[1\mathpunct{:}0\mathpunct{:}0\right]~\mbox{and}~
\left[1\mathpunct{:}0\mathpunct{:}b\right].
\end{equation*}
Denote by \(Q\) its proper transform on \(Y\);
\(Q\) is a \((-1)\) curve on \(Y\)
and therefore it is a section.
\item Let \(E_{7}\) be the exceptional divisor over 
\([1\mathpunct{:}0\mathpunct{:}b]\);
\item Let \(E_{8}\) be the exceptional divisor over 
\([1\mathpunct{:}0\mathpunct{:}c]\);
\end{itemize}
Let \(E_{8}\) be the exceptional divisor over \([1\mathpunct{:}0\mathpunct{:}c]\)
and \(E_{9}\) be the exceptional divisor over \([1\mathpunct{:}0\mathpunct{:}0]\).
Now we can construct our cycles via
\begin{itemize}
\item[(1)] \(\gamma_{1}=L-E_{9}\sim H-E_{5}-E_{6}-E_{9}\);
\item[(2)] \(\gamma_{2}=Q-E_{4}\sim 2H-(E_{1}+\cdots+E_{4})-E_{7}-E_{9}\);
\item[(3)] \(\gamma_{3}=E_{7}-E_{9}\);
\item[(4)] \(\gamma_{4}=E_{8}-E_{9}\);
\end{itemize}
One can easily check that each of them lies in \([D]^{\perp}\)
and they form a
basis of \(\mathrm{H}_{2}(X,\mathbb{Z})\). 
Also the fibre class \([f]\) is
given by \([\gamma_{1}]+[\gamma_{2}]-[\gamma_{4}]\). In the 
present case, we take
\begin{equation*}
\Omega = \frac{x\mathrm{d}y\wedge\mathrm{d}z-
  y\mathrm{d}x\wedge\mathrm{d}z+z\mathrm{d}x\wedge\mathrm{d}y}{x^{2}y}.
\end{equation*}
It is also straightforward to check (parallel
to the computation in previous sections) that the set of period vectors
\begin{equation*}
\left\{(y_{1},\ldots,y_{4})~\Big|~y_{i}:=\int_{[\gamma_{i}]}\Omega,~i=1,\ldots,4\right\}
\end{equation*}
is equal to \(\{(y_{1},\ldots,y_{4})~|~y_{1}+y_{2}-y_{4}=0\}\).
More accurately, one can check
\begin{equation*}
\int_{\gamma_{1}}\Omega=\frac{a^{2}(b+c)+d+ah}{a^{2}},~
\int_{\gamma_{2}}\Omega=-\frac{d}{a^{2}}-\frac{h}{a}-b,~\mbox{and}
\int_{\gamma_{4}}\Omega=c.
\end{equation*}
This proves the surjectivity of the period
map in this case.

\subsubsection{{\bf Type $\mathrm{I}_{2}^{\ast}$}}
\label{subsubsec:type-i2*-construction}

A type \(\mathrm{I}_{2}^{\ast}\) fibre is the \(D_{7}\) configuration.
We begin with \(\underline{D}=\{x^{2}y=0\}\) and consider a cubic \(C\) intersecting
with \(\underline{D}\) at three distinct points on \(\{y=0\}\) but 
tangent to \(\{x=0\}\). Moreover, we require that \(C\)
passes through the unique singularity in \(\underline{D}_{\mathrm{red}}\).
Using the \(\mathrm{PGL}(3,\mathbb{C})\)-action on \(\mathbf{P}^{2}\), 
we may assume that 
\begin{itemize}
\item \(C\cap \{x=0\}=
\left\{[0\mathpunct{:}1\mathpunct{:}0], 
[0\mathpunct{:}0\mathpunct{:}1]\right\}\)
and \(C\) is tangent to \(\{x=0\}\) at \([0\mathpunct{:}1\mathpunct{:}0]\).
\item \(C\cap \{y=0\}=
\left\{[1\mathpunct{:}0\mathpunct{:}0], 
[1\mathpunct{:}0\mathpunct{:}b],
[0\mathpunct{:}0\mathpunct{:}1]\right\}\);
\end{itemize} 
It follows that \(C\) is 
defined by the following equation
\begin{equation*}
z^{2}y+x\left(cy^{2}+dz^{2}+exy-bdxz+gyz\right)=0.
\end{equation*}
We may as well assume that \(e=0\) as before by adding the
defining equation of \(\underline{D}\). 
\begin{equation*}
z^{2}y+x\left(cy^{2}+dz^{2}-bdxz+gyz\right)=0.
\end{equation*}
Choosing a smooth cubic \(C\) of the form above and
blowing-up at those points yield an 
ALG pair \((Y,D)\) of type \(\mathrm{I}_{2}^{\ast}\). 
Now we need to construct a basis of \(\mathrm{H}_{2}(X,\mathbb{Z})\),
i.e., we need to construct four cycles lying in \([D]^{\perp}\)
under the identification
\([D]^{\perp}\subset\mathrm{H}^{2}(Y,\mathbb{Z})\cong\mathrm{H}_{2}(Y,\mathbb{Z})\).
\begin{itemize}
\item Consider a conic tangent to \(\{x=0\}\) 
at \([0\mathpunct{:}1\mathpunct{:}0]\) 
and meeting \(C\) at 
\([0\mathpunct{:}1\mathpunct{:}0]\)
with multiplicity four and passing through 
\([1\mathpunct{:}0\mathpunct{:}0]\).
Explicitly, it is given by  
\begin{equation*}
z^{2} +cxy+gxz=0.
\end{equation*}
which intersects \(\{y=0\}\) at 
\begin{equation*}
\left[1\mathpunct{:}0\mathpunct{:}0\right]~\mbox{and}~
\left[1\mathpunct{:}0\mathpunct{:}-g\right].
\end{equation*}
One checks the conic intersects
\(C\) at \([0\mathpunct{:}1\mathpunct{:}0]\)
with multiplicity four. Indeed, we can solve \(x\sim z^{2}\)
in the local ring at \([0\mathpunct{:}1\mathpunct{:}0]\).
Denote by \(Q\) its proper transform on \(Y\);
\item Consider the tangent of \(C\)
at \([0\mathpunct{:}0\mathpunct{:}1]\).
Explicitly, we have
\begin{equation*}
y+dx=0.
\end{equation*}
Denote by \(L\) the proper transform on \(Y\).
\end{itemize}
Let \(E_{8}\) be the exceptional divisor over \([1\mathpunct{:}0\mathpunct{:}b]\)
and \(E_{9}\) be the exceptional divisor over \([1\mathpunct{:}0\mathpunct{:}0]\).
Now we can construct our cycles via
\begin{itemize}
\item[(1)] \(\gamma_{1}=L-E_{7}\sim H-E_{5}-E_{6}-E_{7}\);
\item[(2)] \(\gamma_{2}=Q-E_{9}\sim 2H-(E_{1}+\cdots+E_{4})-2E_{9}\);
\item[(3)] \(\gamma_{3}=E_{8}-E_{9}\);
\end{itemize}
One can easily check that each of them lies in \([D]^{\perp}\)
and they form a
basis of \(\mathrm{H}_{2}(X,\mathbb{Z})\). 
Also the fibre class \([f]\) is
given by \([\gamma_{1}]+[\gamma_{2}]-[\gamma_{3}]\). In the 
present case, we take
\begin{equation*}
\Omega = \frac{x\mathrm{d}y\wedge\mathrm{d}z-
  y\mathrm{d}x\wedge\mathrm{d}z+z\mathrm{d}x\wedge\mathrm{d}y}{x^{2}y}.
\end{equation*}

\begin{lem}
We have
\begin{equation*}
\int_{[\gamma_{1}]}\Omega = -(b+g),~
\int_{[\gamma_{2}]}\Omega = g,~\mbox{and}~
\int_{[\gamma_{3}]}\Omega = -b.
\end{equation*}
\end{lem}
\begin{proof}
Using the affine coordinates \(u=x/z\) and \(v=y/z\),
we may re-write
\begin{equation*}
\Omega = \frac{\mathrm{d}u\wedge\mathrm{d}v}{u^{2}v}
\end{equation*}
and the equation of \(C\) and the tangent line are
given by
\begin{equation*}
v+du-bdu^{2}+guv+cuv^{2}~\mbox{and}~v+du.
\end{equation*}

Let \((u,v=\alpha u)\) be the coordinates on (an affine chart of) the blow-up (\(\alpha\) 
is the affine coordinate on the exceptional divisor). Then \(\Omega\)
is transformed into
\begin{equation*}
\Omega = \frac{\mathrm{d}u\wedge\mathrm{d}\alpha}{u^{2}\alpha}
\end{equation*}
and the proper transform of \(C\) and the tangent line are
given by
\begin{equation*}
\alpha+d-bdu+g\alpha u+c\alpha^{2}u^{2}~\mbox{and}~\alpha+d.
\end{equation*}
Now we have to blow-up at \(u=0\) and \(\alpha=-d\).
Let \(\alpha'=\alpha+d\). The equations above become
\begin{equation*}
\alpha'-bdu+g(\alpha'-d) u+c(\alpha'-d)^{2}u^{2}~\mbox{and}~\alpha'.
\end{equation*}
Moreover, we have
\begin{equation*}
\Omega = \frac{\mathrm{d}u\wedge\mathrm{d}\alpha'}{u^{2}(\alpha'-d)}.
\end{equation*}
Now we perform the blow-up via \(\alpha'=\alpha'\) and \(u=\alpha'\beta\). 
\begin{equation*}
\Omega = \frac{\mathrm{d}\beta\wedge\mathrm{d}\alpha'}{\alpha'\beta^{2}(\alpha'-d)}.
\end{equation*}
Taking the residue around \(\{\alpha'=0\}\), we obtain a one-form 
\begin{equation*}
-\frac{1}{d}\frac{\mathrm{d}\beta}{\beta^{2}}
\end{equation*}
on \(\mathbf{P}^{1}\)
with a double pole at \(\beta=0\). 
The proper transform of \(C\) becomes
\begin{equation*}
1 - (bd+gd)\beta + \mbox{higher order terms}.
\end{equation*}
Therefore, we have
\begin{equation*}
\int_{[\gamma_{1}]}\Omega = -(b+g).
\end{equation*}
The other cases are similar.
\end{proof}

It is also straightforward to check (parallel
to the computation in previous sections) that the set of period vectors
\begin{equation*}
\left\{(y_{1},y_{2},y_{3})~\Big|~y_{i}:=\int_{[\gamma_{i}]}\Omega,~i=1,2,3\right\}
\end{equation*}
is equal to \(\{(y_{1},y_{2},y_{3})~|~y_{1}+y_{2}-y_{3}=0\}\).
This proves the surjectivity of the period
map in this case.

\subsubsection{{\bf Type $\mathrm{I}_{3}^{\ast}$}}
\label{subsubsec:type-i3*-construction}

A type \(\mathrm{I}_{3}^{\ast}\) fibre is the \(D_{7}\) configuration.
We again begin with \(\underline{D}=\{x^{2}y=0\}\) and consider a cubic \(C\) intersecting
with \(\underline{D}\) as follows.
\begin{itemize}
\item \(C\cap \{x=0\}=
\left\{[0\mathpunct{:}1\mathpunct{:}0], 
[0\mathpunct{:}0\mathpunct{:}1]\right\}\)
and \(C\) is tangent to \(\{x=0\}\) at \([0\mathpunct{:}1\mathpunct{:}0]\).
\item \(C\cap \{y=0\}=
\left\{[1\mathpunct{:}0\mathpunct{:}0], 
[0\mathpunct{:}0\mathpunct{:}1]\right\}\) 
and \(C\) is tangent to \(\{y=0\}\)
at \([0\mathpunct{:}0\mathpunct{:}1]\);
\end{itemize} 
It follows that \(C\) is 
defined by the following equation
\begin{equation*}
z^{2}y+x\left(ay^{2}+dxz+eyz\right)=0.
\end{equation*}
Picking a smooth cubic \(C\) of the form above and
blowing-up at those points yield a 
ALG pair \((Y,D)\) of type \(\mathrm{I}_{3}^{\ast}\). 
We need to construct a basis of \(\mathrm{H}_{2}(X,\mathbb{Z})\),
i.e., we need to construct two cycles lying in \([D]^{\perp}\)
under the identification.
\begin{itemize}
\item Consider a conic intersecting \(C\) at 
\([0\mathpunct{:}1\mathpunct{:}0]\)
with multiplicity four.
Explicitly, we may pick  
\begin{equation*}
axy+z^{2}+exz=0.
\end{equation*}
which intersects \(\{y=0\}\) at 
\(\left[1\mathpunct{:}0\mathpunct{:}0\right]\) and 
\(\left[1\mathpunct{:}0\mathpunct{:}-e\right]\).

Denote by \(Q_{1}\) its proper transform on \(Y\);
\item Consider another conic intersecting 
\(C\) at \([0\mathpunct{:}0\mathpunct{:}1]\) with 
multiplicity greater than or equal to four and passing through 
\([0\mathpunct{:}1\mathpunct{:}0]\).
Explicitly, we can take
\begin{equation*}
dx^{2}+zy+exy=0.
\end{equation*}
Denote by \(Q_{2}\) its proper transform on \(Y\);
\item Let \(L\) be the proper transform of the line
connecting \([1\mathpunct{:}0\mathpunct{:}0]\)
and \([0\mathpunct{:}0\mathpunct{:}1]\).
\end{itemize}
Let \(E_{9}\) be the exceptional divisor over \([1\mathpunct{:}0\mathpunct{:}0]\).
Now we can construct our cycles via
\begin{itemize}
\item[(1)] \(\gamma_{1}=Q_{1}-E_{9}\sim 2H-(E_{1}+\cdots+E_{4})-2E_{9}\);
\item[(2)] \(\gamma_{2}=Q_{2}-L\sim 2H-(E_{5}+\cdots+E_{8})-E_{1}-(H-E_{1}-E_{9})\);
\end{itemize}
One can easily check that each of them lies in \([D]^{\perp}\)
and they form a
basis of \(\mathrm{H}_{2}(X,\mathbb{Z})\). 
Also the fibre class \([f]\) is
given by \([\gamma_{1}]+[\gamma_{2}]\). In the 
present case, we take
\begin{equation*}
\Omega = \frac{x\mathrm{d}y\wedge\mathrm{d}z-
  y\mathrm{d}x\wedge\mathrm{d}z+z\mathrm{d}x\wedge\mathrm{d}y}{x^{2}y}.
\end{equation*}
It is also straightforward to check (parallel
to the computation in previous sections) that the set of period vectors
\begin{equation*}
\left\{(y_{1},y_{2})~\Big|~y_{i}:=\int_{[\gamma_{i}]}\Omega,~i=1,2\right\}
\end{equation*}
is equal to the set \(\{(y_{1},y_{2})~|~y_{1}+y_{2}=0\}\).
This proves the surjectivity of the period
map in this case.

\subsubsection{{\bf Type $\mathrm{I}_{4}^{\ast}$}}
\label{subsubsec:type-i4*-construction}
A type \(\mathrm{I}_{4}^{\ast}\) fibre is the \(D_{8}\) configuration.
Suppose \((Y,D)\) is an \(\mathrm{ALG}^{\ast}\) pair of type \(\mathrm{I}_{4}^{\ast}\).
One easily checks that \(\mathrm{H}_{2}(X,\mathbb{Z})\cong\mathbb{Z}\)
is generated by the homology class of a fibre,
which is represented by a holomorphic curve.
Consequently, similar to the case \(\mathrm{II}^{\ast}\),
the period map must be constant (in fact the zero map).
To make the treatment comprehensive, we
will outline the construction of an \(\mathrm{ALG}^{\ast}\)
pair of type \(\mathrm{I}_{4}^{\ast}\).

We begin with a line and a conic tangent at a point \(p\).
Again we may assume that the line is given by \(\{x=0\}\) and
\(p=[0\mathpunct{:}0\mathpunct{:}1]\).
Consider a smooth cubic which intersects the conic
with multiplicity five. In which case,
the cubic intersects both the line and the conic
at a point other than \(p\). Explicity, we can take for instance
the conic to be
\begin{equation*}
\{xz-y^{2}=0\}
\end{equation*}
and the cubic to be
\begin{equation*}
\{xz^{2}-zy^{2}-x^{3}=0\}.
\end{equation*}
We can blow-up all the intersection points (nine points in total)
to achieve a rational elliptic surface with an \(\mathrm{I}_{4}^{\ast}\)
configuration.
\begin{figure}
\includegraphics[scale=0.8]{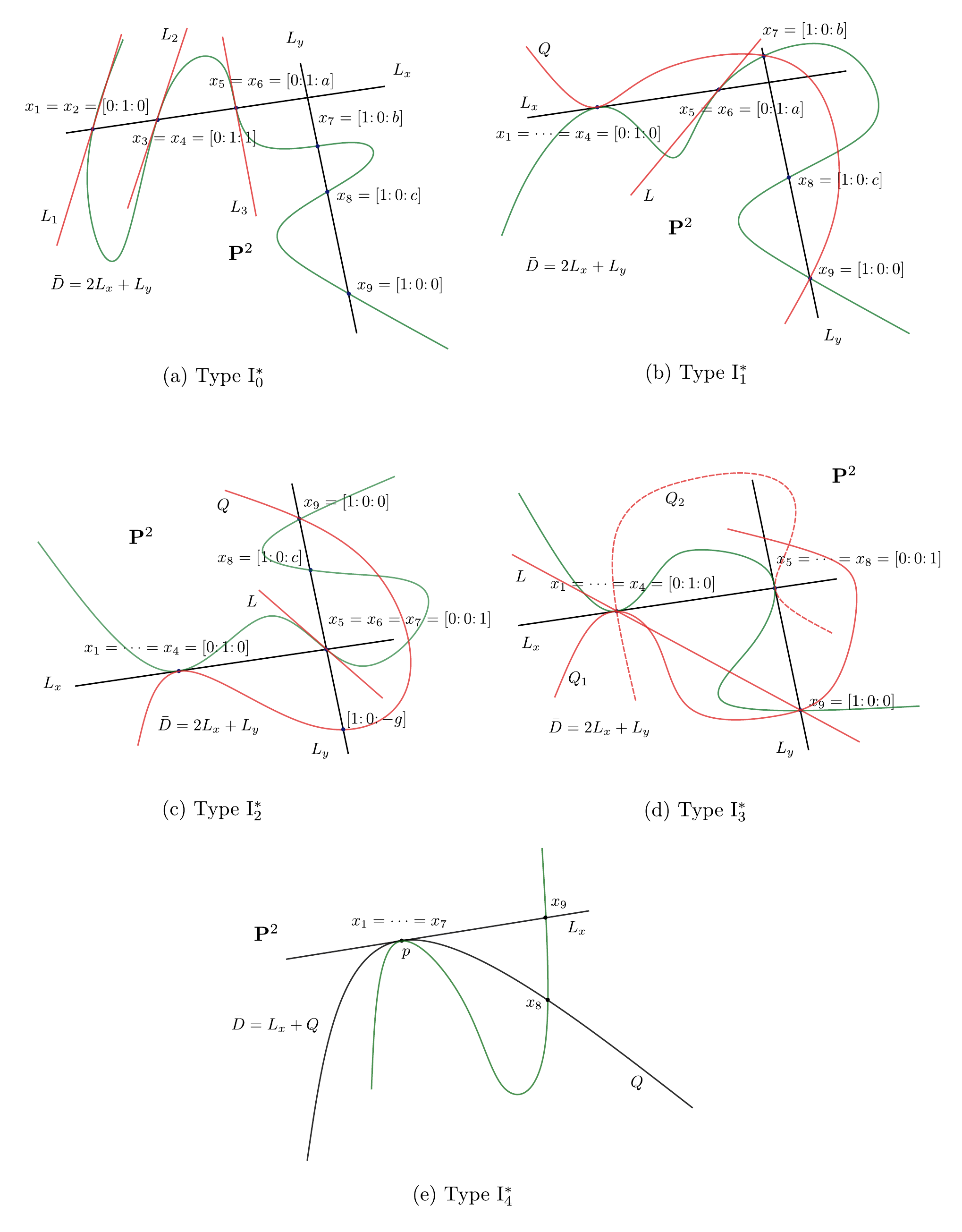}
\caption{The constructions of \(\mathrm{ALG}^{\ast}\)
pairs of type \(\mathrm{I}_{1}^{\ast}\) to
\(\mathrm{I}_{4}^{\ast}\) and ALG pair
of type \(\mathrm{I}_{0}^{\ast}\). 
Here again the black lines stand for the singular divisor \(\underline{D}\),
the green lines stand for the cubic \(C\) in the corresponding pencil, and the
red lines (solid/dashed) indicate the cycles we use to construct \((-1)\) curves on \(Y\).}
\end{figure}

\begin{rk}
\label{rk:std}
Let \((Y,D)\) be an \(\mathrm{ALG}\) or \(\mathrm{ALG}^{\ast}\) pair.
In each case above, our construction gives a deformation family
\(\pi\colon (\mathcal{Y},\mathcal{D})\to \mathcal{M}\) of \((Y,D)\) together with a section of 
\(\pi_{\ast}\Omega^{2}_{\mathcal{Y}\slash\mathcal{M}}(\mathcal{D})\)
with a fixed normalization; it
is the pullback of a fixed section 
\begin{equation}
\Omega\in \mathrm{H}^{0}(\mathbf{P}^{2},\Omega_{\mathbf{P}^{2}}^{2}(\underline{D})) = 
\mathrm{H}^{0}(Y,\Omega_{Y}^{2}(D)).
\end{equation}
This normalization will
become essential in the later subsection when we discuss the
\(\mathrm{ALG}\) or \(\mathrm{ALG}^{\ast}\) gravitational instantons, i.e.,
when metrics are involved.
\end{rk}

\subsection{\texorpdfstring{$\mathrm{ALG}$}{ALG} and \texorpdfstring{$\mathrm{ALG}^{\ast}$}{ALG*} gravitational instantons}
\label{subsec:alg-alg*-gra-ins}
We will start with introducing the models for $\mathrm{ALG}$ and $\mathrm{ALG}^{\ast}$ gravitational instantons. 
A model for $\mathrm{ALG}$ gravitational instantons is determined by $(\beta,\tau,L,R)$, 
where $R,L>0$ and $\beta,\tau$ is chosen from Table \ref{table: ALG} below. 

\begin{table}[htbp!]
\centering
\begin{spacing}{2}
\begin{tabular}{|l|c|c|c|c|c|c|c|} 
	\hline
	$\infty$   &  $\mathrm{I}_0^{\ast}$     &  $\mathrm{II}$      &  $\mathrm{III}$     & $\mathrm{IV}$    & $\mathrm{IV}^{\ast}$ & $\mathrm{III}^{\ast}$ & $\mathrm{II}^{\ast}$ \\
	\hline
	$\beta$    & $1/2$     & $1/6$     & $1/4$     & $1/3$ &  $2/3$ & $3/4$  & $5/6$  \\
	 \hline
	$\tau$   & any     & $e^{2\pi i/3}$    &  $i$  & $e^{2\pi i/3}$ &  $e^{2\pi i/3}$ &  $i$    & $e^{2\pi i/3}$    \\
	\hline
\end{tabular}
\end{spacing}
\caption{\label{table: ALG}}
\end{table}
For each triple $(\beta,\tau,L)$ chosen, denote by $X_{mod}=X_{mod}(\beta,\tau,L,R)$ the complex manifold
  \begin{align*}
     \left\{(u,v)\in \mathbb{C}\oplus \mathbb{C}~|~\operatorname{Arg}(u)\in[0,2\pi \beta]~\mbox{and}~|u|>R \right\}/ \sim,
  \end{align*} where the equivalence relation is given by 
  \begin{align*}
   &  (u,v)\sim (u,v+(m+n\tau)), \mbox{ for } (m,n)\in \mathbb{Z}^2,\\ 
    & (u,v) \sim (e^{2\pi i\beta}u,e^{-2\pi i\beta}v).
  \end{align*} The hyperK\"ahler triple is given by 
  \begin{align*}
     \omega_{mod}= \frac{\sqrt{-1}L^{2}}{2}(\mathrm{d}u\wedge \mathrm{d}\bar{u}+\mathrm{d}v\wedge \mathrm{d}\bar{v}), \hspace{4mm} 
     \Omega_{mod}= L^{2}\mathrm{d}u\wedge \mathrm{d}v.  
  \end{align*} 
By definition, there is a natural elliptic fibration structure $u^{\frac{1}{\beta}}\colon X_{mod}\rightarrow \{z~|~|z|>R^{\frac{1}{\beta}}\}\subseteq\mathbb{C}$ with torus fibres of area $L^2\operatorname{Im}(\tau)$. Moreover, one can fill in the fibre at infinity to partially compactify $X_{mod}$ and the singular configuration of the fibre at infinity is described in Table \ref{table: ALG} and always with monodromy of finite order. 

On the other hand, a model of $\mathrm{ALG}^{\ast}$ gravitational instanton is determined by $\nu\in \mathbb{N}$ and $R,\epsilon>0$. For each pair $(\nu,\epsilon)$, we denote model by $X^*_{mod}=X^*_{mod}(\nu,\epsilon,R)$ as a complex manifold is given by
   \begin{align*}
      \{(u,v)\in \mathbb{C}\oplus \mathbb{C}~|~u\neq 0, |u|<R\}/\sim, 
   \end{align*} where the equivalence equation is given by 
   \begin{align*}
     & (u,v)\sim (u^2,uv) \\
     & (u,v)\sim (u, v+m+n\frac{\nu}{\pi i}\log{u}), \mbox{ for }(m,n)\in \mathbb{Z}^2.  
   \end{align*} The Ricci-flat metric and the corresponding holomorphic volume form is given by 
   \begin{align*}
      &\omega_{mod}^*=i\frac{\nu |\log{|u|}|}{\pi \epsilon}\frac{\mathrm{d}u\wedge \mathrm{d}\bar{u}}{|u|^4}+\frac{i}{2}\frac{\pi \epsilon}{\nu |\log{|u|}|}(\mathrm{d}v-\frac{1}{i}\frac{\mbox{Im}(v)\mathrm{d}u}{u|\log{|u|}|})\wedge (\overline{\mathrm{d}v-\frac{1}{i}\frac{\mbox{Im}(v)\mathrm{d}u}{u|\log{|u|}|}}), \\
      &\Omega_{mod}^*= u^{-2}\mathrm{d}u\wedge \mathrm{d}v.
   \end{align*} There is also an natural elliptic fibration structure $u^{\frac{1}{2}}\colon X^*_{mod}\rightarrow \{ |0<|u|^{\frac{1}{2}}<R^{\frac{1}{2}} \} \subseteq \mathbb{C}$ and one can extends the fibration over the puncture by adding an $\mathrm{I}_{\nu}^*$-fibre.

  \begin{defn}
    We say that a gravitational instanton $(X,\omega,\Omega)$ is of type $\mathrm{ALG}(\beta,\tau,L)$ (or $\mathrm{ALG}$ for simplicity) if the Calabi ansatz $(X_{\mathcal{C}},\omega_{\mathcal{C}},\Omega_{\mathcal{C}})$ in Definition \ref{def: ALH*} is replaced by 
    \begin{equation*}
    (X_{mod}(\beta,\tau,L,R),\omega_{mod},\Omega_{mod})~\mbox{for some}~R>0.
    \end{equation*} 
    We also define marked $\mathrm{ALG}$ gravitational instanton similar to Definition \ref{def: marked ALH*}. We will denote the set of marked $\mathrm{ALG}(\beta,\tau,L)$ gravitational instantons by $\mathrm{mALG}(\beta,\tau,L)$. We will define (marked) $\mathrm{ALG}^{\ast}(\nu,\epsilon)$ gravitational instantons and $\mathrm{mALG}^{\ast}(\nu,\epsilon)$ similarly. 
  \end{defn}
 \begin{rk}
 	Here the definition seems different from the one in \cite{CVZ2} a priori. However, the definitions of $\mathrm{ALG}$ gravitational instantons are equivalent from \cite{Hein}*{(3.10)} and different type of $D$ corresponds different value of choice of $\beta$ in \cite{CVZ2}, which is a discrete parameter. The definitions of $\mathrm{ALG}^{\ast}$ gravitational instantons are equivalent by \cite{CV}*{Proposition 2.3}.
 \end{rk}

With the above definition, there are natural invariants of the $\mathrm{ALG}$ gravitational instantons given by the cohomology classes of the hyperK\"ahler triple. The set of possible cohomology classes are called the period domain of the gravitational instantons. In the cases of $\mathrm{ALG}$ and $\mathrm{ALG}^{\ast}$ gravitational instantons, the period domains are described by Chen--Viaclovsky--Zhang \cite{CVZ2}. For the $\mathrm{ALG}$ case, we first fixed $\beta,\tau,L$ and a reference $\mathrm{ALG}(\beta,\tau,L)$ gravitational instanton $(X_0,\omega_0,\Omega_0)$. The period domain $\mathcal{P}\Omega(\beta,\tau,L)$ is a subset of $\mathrm{H}^2(X_0,\mathbb{R})\times \mathrm{H}^2(X_0,\mathbb{C})$ consisting of pairs $([\omega],[\Omega])$ satisfying the following conditions:
\begin{enumerate}
	\item if $[C]\in \mathrm{H}_2(X_0,\mathbb{Z})$, $[C]^2=-2$, then $|[\omega]\cdot[C]|^2+|[\Omega]\cdot[C]|^2\neq 0$;
	\item $[\Omega]\cdot[F]=0$, where $[F]\in \mathrm{H}_2(X_0,\mathbb{Z})$ is the homology class of the elliptic fibre;
	\item $[\omega]\cdot [F]=L^2\operatorname{Im}(\tau)$. 
\end{enumerate}
The period domain $\mathcal{P}\Omega(\nu,\epsilon)$ of $\mathrm{ALG}^{\ast}(\nu,\epsilon)$ gravitational instantons are defined similarly except the last condition is replaced by $[\omega]\cdot [F]=\epsilon$. Then the period map of marked $\mathrm{ALG}(\beta,\tau,L)$ gravitational instanton is defined by 
\begin{align}
    \mathcal{P}(\beta,\tau,L)\colon \mathrm{mALG}(\beta,\tau,L)&\rightarrow \mathcal{P}\Omega(\beta,\tau,L) \\
     (X,\omega,\Omega,\alpha)  &\mapsto (\alpha^*[\omega],\alpha^*[\Omega]).
\end{align} We define the period map $\mathcal{P}(\nu,\epsilon)$ for $\mathrm{ALG}^{\ast}(\nu,\epsilon)$ gravitational instantons similarly.  

The goal of the section is to prove the surjectivity of the period maps of $\mathrm{ALG}$ and $\mathrm{ALG}^{\ast}$ gravitational instantons, conjectured by Chen--Viaclovsky--Zhang \cite{CVZ2}*{Conjecture 7.8}.
\begin{thm}\label{main2}
	The period maps $\mathcal{P}(\beta,\tau,L)$ and $\mathcal{P}(\nu,\epsilon)$ are surjective. 
\end{thm}

  Similar to the $\mathrm{ALH}^{\ast}$ gravitational instantons, we have the following 
  uniformization results for $\mathrm{ALG}$ and $\mathrm{ALG}^{\ast}$ gravitational instantons.
  \begin{thm}[{\cite{CC1}*{Theorem 1.2} and \cite{CV}*{Theorem 1.5}}]
  	 Any $\mathrm{ALG}$ (or $\mathrm{ALG}^{\ast}$) gravitational instanton can be compactified 
	 to a rational elliptic surfaces by adding a singular fibre of finite monodromy (or of type $\mathrm{I}_{\nu}^*$). 
  \end{thm}
 \begin{rk}
 	From the Persson's classification of singular configurations in rational elliptic surfaces \cite{Per}, 
	one can only have $\nu\leq 4$ for $\mathrm{ALG}^{\ast}$ gravitational instantons \cite{CV}.  
 \end{rk}

Therefore, we will follow the method similar to the proof of the surjectivity of the period map for $\mathrm{ALH}^{\ast}$ gravitational instantons to prove Theorem \ref{main2}. We already proved the surjectivity of the $(2,0)$-form for rational elliptic surfaces with a prescribed fibre with finite monodromy or of type $\mathrm{I}_{\nu}^{\ast}$ in Theorem \ref{thm: surj (2,0)-form}, and we will later prove that every cohomology class of the complement of the prescribed fibre in the rational elliptic surface can support a Ricci-flat metric up to monodromy (see Theorem \ref{thm: Hein} and Lemma \ref{lifting}).

\subsection{Surjectivity of the period maps for \texorpdfstring{$\mathrm{ALG}$}{ALG} and \texorpdfstring{$\mathrm{ALG}^{\ast}$}{} gravitational instantons}
\label{subsec:sur-alg-alg*}
   We next modify a theorem of Hein \cite{Hein}*{Theorem 1.3}. Let $Y$ be a rational elliptic surface with a fibre $D$ of finite order monodromy. Denote by $X=Y\setminus D$ and by $p\colon X\rightarrow B\cong\mathbb{C}$ the restriction of the elliptic fibration structure from $Y$. We fix a holomorphic coordinate $u$ on a neighborhood of the base such that the singular fibre $D$ is located at $u=0$. Finally, let $U_{r}=\{u\in B~|~|u|<r\}$. 
\begin{thm} \label{thm: Hein}
	Let $\omega_0$ be any K\"ahler metric on $X=Y\setminus D$ such that $\int_{X}\omega_0^2<\infty$\footnote{For our purpose, we will only take those K\"ahler forms on $Y$ and restrict to $X$.}. Given $\alpha>0$, there exists a Ricci-flat metric $\omega$ such that $[\omega]=[\omega_0]$ and $\omega^2=\alpha\Omega\wedge \bar{\Omega}$ for a fixed meromorphic volume form $\Omega$ with a simple pole along $D$. Moreover, one has 
	\begin{align*}
	   \| \nabla^k(\omega-\omega_{mod})\|_{g_{mod}}\lesssim O(r^{-k-2})
	\end{align*} for any $k\in \mathbb{N}$. 
\end{thm}
\begin{proof}
	From \cite{Hein}*{Eq.~(3.25)}, Hein constructed a background K\"ahler form $\omega_a$ on $X$ such that 
	\begin{enumerate}
		\item $[\omega_a]=[\omega_0]\in \mathrm{H}^2(X,\mathbb{R})$. 
		\item There exists $0<r_1<r_2$ such that 
		\begin{itemize}
			\item $\omega_a=\omega_0$ in $U_{r_2}^{c}$. 
			\item $\omega_a=T^*\omega_{sf,\epsilon}(\alpha)$ on $U_{r_1}$, where $T$ is the fibrewise translation by a holomorphic section over $U_{r_1}$. In particular, $\omega_a^2=\alpha\Omega\wedge\bar{\Omega}$ on $p^{-1}(U_{r_1})$. 
		\end{itemize}      
	\end{enumerate}
	We will modify $\omega_a$ such that it satisfies the integrability condition 
	\begin{align*}
	\int_{X}\omega_a^2-\alpha\Omega\wedge \bar{\Omega}=0. 
	\end{align*}
	For $0<r<s<r_1$, we define $\beta_{r,s}$ to be a $2$-form on $B$ such that $\beta_{r,s}=\frac{1}{2}\chi(|u|)f(|u|)\mathrm{d}u\wedge \mathrm{d}\bar{u}$, where $\chi\colon\mathbb{R}_+\rightarrow \mathbb{R}_+$ is a cut-off function with $\chi(t)\in [0,1]$ such that $\chi(t)=1$ on $U_{s}\setminus U_{r}$ 
 and 
	\begin{align*}
	    f(t)= \begin{cases}
	    \displaystyle\frac{\nu |\log{t}|}{2\pi \epsilon t^4}, &\mbox{ if $D$ is of type $\mathrm{I}_{\nu}^*, \nu>0$}, \\
	    \\
	    \displaystyle 1/t^{4}, &\mbox{ if $D$ is of finite monodromy.}
	    \end{cases}
	\end{align*}
	
	By a direct calculation, we have $\omega_a\pm \beta_{r,s}$ which is again a K\"ahler form. Notice that by another straightforward calculation,
	$\int_{X}\omega_a\wedge \beta_{r,s}\rightarrow \infty$ as $r\rightarrow 0$. Thus we have 
	\begin{align*}
	\int_{X}(\omega_a+\beta_{r,s})^2-\alpha \Omega\wedge \bar{\Omega}\rightarrow \infty, \mbox{ for $r\to 0$},\\
	\int_{X}(\omega_a^2-\beta_{r,s})^2-\alpha \Omega\wedge \bar{{\Omega}}\rightarrow -\infty, \mbox{ for $r\to 0$}.
	\end{align*} Then there exists $t'\in [-1,1]$ such that $\omega_a+t'\beta_{r,s}$ achieves the integrability condition for some $r$ by intermediate value theorem.    	
	
	With the integrability condition, the existence of the Ricci-flat metric in the same cohomology class (actually in the same Bott--Chern cohomology class) is guaranteed by \cite{TY}*{Theorem 1.1}. 
	Then \cite{Hein}*{Proposition 2.9} provides the decay to the model metrics
		\begin{align*}
	\|\nabla^k(\omega-\omega_{mod})\|_{g_{mod}}\lesssim O(r^{-k-\mathfrak{n}})
	\end{align*} for any $k\in \mathbb{N}$. Here $\mathfrak{n}$ can be taken to be $2$ if $D$ is of type $\mathrm{II}$,
	$\mathrm{III}$, $\mathrm{IV}$, or $\mathrm{I}_{\nu}^{\ast}$ and the theorem is proved. For the case when $D$ is of the type $\mathrm{II}^{\ast}$, $\mathrm{III}^{\ast}$, or $\mathrm{IV}^{\ast}$, \cite{CV}*{Proposition 5.1} showed that there exists a gravitational instanton with hyperK\"ahler triple of the same cohomology class and the required asymptotic. 
\end{proof}	
To prove the surjectivity of the period map (Theorem \ref{main2}), we also need the following lemma.
\begin{lem} \label{lifting}
	Given $[\omega]\in \mathrm{H}^2(X,\mathbb{R})$ such that $[\omega]$ is positive on every holomorphic curve in $X$, then there exists a K\"ahler class $[\omega_Y]\in \mathrm{H}^2(Y,\mathbb{R})$ such that $[\omega_Y]|_X=[\omega]$. 
\end{lem}
\begin{proof}
   From the dual of \eqref{eq:alg-alg*-es}, any two liftings of $[\omega]$ in $\mathrm{H}^2(Y,\mathbb{R})$ are differed by a linear combination of $\mbox{PD}([D_i])$. Recall that a cohomology class $[\omega_Y]\in \mathrm{H}^2(Y,\mathbb{R})$ is K\"ahler if it is positive on every holomorphic curve in $Y$ by \cite{DP}*{Theorem 0.1}.  Holomorphic curves in $Y$ are either those avoid $D$, those has positive intersection with $D$ or the components of $D$. 
   Choose any lifting $[\omega'_Y]\in \mathrm{H}^2(Y,\mathbb{R})$ of $[\omega]$ is positive on the holomorphic curves of the first kind. 
   
   For the case $D$ is not of type $\mathrm{IV}$, the dual intersection complex of $D$ is a tree. We choose a root and label the components of $D$ with respect to the partial ordering given by the distance to the root, say $D_1,\ldots, D_n$. In the case when $D$ is of type $\mathrm{IV}$, we will simply choose arbitrary labeling. Then we can inductive solve $a_i$ such that $\big([\omega'_Y]+\sum_{i} a_i\mbox{PD}([D_i])\big).[D_j]=\epsilon>0$ for $j=1,\ldots, n-1$. Since $[\omega'_Y].[D]>0$, we have $\big([\omega'_Y]+\sum_{i} a_i\mbox{PD}([D_i])\big).[D_{n}]>0$ by choosing $\epsilon$ small enough.  
   
   From \cite{Bo}*{Proposition 6.2}, the cone of effective curves of $Y$ is the convex hull of a set of extremal rays given by rational curves and possibly $[D]$, accumulating at most to $\mathbb{R}_+[D]$.    
   Since $[\omega'_Y]+\sum_i a_i\mbox{PD}([D_i])$ is positive on $[D]$, we have $[\omega_Y]=[\omega'_Y]+\sum_i a_i\mbox{PD}([D_i])+t\mbox{PD}([D])$ is also positive on the curves of the second kind for $t\gg 0$. Thus, $[\omega_Y]$ is a K\"ahler class we are looking for. 
\end{proof}

\begin{proof}[Proof of Theorem \ref{main2}]
   	The proof is similar to the proof of Theorem \ref{main1},
   	where Theorem \ref{thm: surj (2,0)-form}, and Theorem \ref{thm: Hein} are the replacements for Theorem \ref{thm:surjectivity-of-periods}, and Theorem \ref{TY} \cite{TY}.
\end{proof}	

Finally, we comment on a Torelli theorem of the pairs $(Y,D)$. It is known that the periods of the holomorphic $(2,0)$-form on $X=Y\setminus D$ determined the isomorphism class of the pair $(Y,D)$ when $D$ is an $\mathrm{I}_k$-fibre \cite{GHK} and when $D$ is smooth \cite{Ba} (see also Appendix \ref{app}). However, it seems that there is less study when $D$ has components with multiplicities. Here we take the advantage of the Torelli theorem of gravitational instantons of $\mathrm{ALG}$ or $\mathrm{ALG}^{\ast}$ \cite{CVZ2} and give an optimal result when $D$ is not reduced. 
\begin{prop}\label{inj}
	Assume that $D$ is of type $\mathrm{II}$, $\mathrm{III}$, $\mathrm{IV}$ or $\mathrm{I}^{\ast}_{\nu}$ with $\nu\in \{0,1,2,3,4\}$. Let $(Y_1,D)$ and $(Y_2,D)$ be two pairs of rational elliptic surfaces with prescribed singular fibres. Let $\Omega_i$ be the meromorphic $(2,0)$-form on $Y_i$ with as simple pole along $D$ with the residue of $\Omega_i$ being fixed and there exists a diffeomorphism $f\colon X_2\rightarrow X_1$ such that $f^*\Omega_1,\Omega_2$ have the same periods on $X$. Then there exists an isomorphism $(Y_1,D)\cong (Y_2,D)$ as pairs. 
\end{prop}
\begin{proof}
	 From Theorem \ref{thm: Hein} and Lemma \ref{lifting}, there exists Ricci-flat metrics $\omega_i$ on $X_i$ such that $f^*[\omega_1]=[\omega_2]$. Then by Torelli theorem of $\mathrm{ALG}$ (or $\mathrm{ALG}^{\ast}$) gravitational instantons \cite{CVZ2}, one may modify the diffeomorphism $f$ such that $f^*\omega_1=\omega_2$ and $f^*\Omega_1=\Omega_2$. In particular, $f$ is a biholomorphism and thus $Y_1$ and $Y_2$ are birational to each other. Therefore, there exist a compact complex surface $Y$ and birational morphisms $f_1\colon Y\rightarrow Y_1$ and $f_2\colon Y\rightarrow Y_2$ such that $f_i$ are compositions of sequences of simple blow-ups. Since $ Y_1\setminus X_1\cong Y_2\setminus X_2$ both biholomorphic to $D$, $f_1$ and $f_2$ must undo each other. In other words, $f\colon X_2\rightarrow X_1$ can be extended to a biholomorphism $Y_2\rightarrow Y_1$, sending the one boundary divisor isomorphically to another. 
\end{proof}

\begin{rk} \begin{enumerate}
		\item  Here the condition fixing the residue of $\Omega_i$ is the substitution of the normalization condition in \cite{Fr}*{p.~22}.
		\item The injectivity of the period map is only true when the metric is asymptotic to the model of order $2$ when $D$ is of 
		type $\mathrm{II}^{\ast}$, $\mathrm{III}^{\ast}$, or $\mathrm{IV}^{\ast}$ \cites{CC3,CV} and thus the argument of the proof for Proposition \ref{inj} breaks down in these cases. This is because that there are isotrivial degenerations of rational elliptic surfaces with such prescribed fibres. 
	\end{enumerate}
\end{rk}	

\appendix

\section{Torelli theorem for log Calabi--Yau surfaces} \label{app}
The following Torelli theorem is implicitly hidden in the work of \cites{M,Fr} and is known to experts. However, the authors cannot find the exact statement in the literature and so we include the proof here to make the article self-contained. 
\begin{thm} \label{Torelli log CY}
	Consider two pairs consisting of a weak del Pezzo surface\footnote{More generally it is true for successive blow-ups of $\mathbb{P}^2$ on a smooth irreducible anti-canonical divisor.} and a smooth anti-canonical divisor $(Y,D)$ and $(Y',D')$. Assume that there exists a deformation family of pairs  $(\mathcal{Y},\mathcal{D})\rightarrow B$ such that both $(Y,D)$ and $(Y',D')$ are fibres. Denote by $\mu\colon \mathrm{H}^2(X,\mathbb{C})\rightarrow \mathrm{H}^2(X',\mathbb{C})$ the isomorphism via some parallel transport, where $X=Y\setminus D$ and $X'=Y'\setminus D'$. 
	If there exist meromorphic volume forms $\Omega$  on $Y$ and $\Omega'$ on $Y'$, with simple poles along $D$ and $D'$ (respectively) such that $\mu([\Omega])=[\Omega']$, then there exists an isomorphism of pairs $f\colon (Y',D')\cong (Y,D)$\footnote{In general, $f^*$ and $\mu$ may differ by reflection of $(-2)$-curves and might not coincide.}   
\end{thm}

First we review some lattice theory. Denote by $\mathbb{Z}^{1,n}$ the lattice generated by $h,e_1,\ldots, e_n$ with the pairing $h^2=1$, $h\cdot e_i=0$, and $e_i\cdot e_j=-\delta_{ij}$. Set $f=3h-\sum_i e_i$, $\alpha_0=e_0-e_1-e_2-e_3$, and $\alpha_i=e_i-e_{i+1}$. Let $\mathbb{L}_n\subseteq \mathbb{Z}^{1,n}$ be the sublattice generated by $\alpha_i$'s. If $Y$ is a blow-up of $\mathbf{P}^2$ at $n$ points, then $\operatorname{Pic}(Y)\cong \mathbb{Z}^{1,n}$. If $D$ is a smooth irreducible anti-canonical divisor of $Y$ and $\Lambda(Y,D)$ denotes the sublattice of $\operatorname{Pic}(Y)$ with zero pairing with $[D]$, then $\Lambda(Y,D)\cong \mathbb{L}_n$. 

Consider the data $(Y,D)$, $\Omega$, and a homology class $\delta \in \mathrm{H}_2(Y,\mathbb{Z})$ such that $\delta\cdot D=0$. From the long exact sequence \eqref{eq:long-exact-sequence}, we can find a representative $\tilde{\delta}$ of $\delta$ contained in $X$ and thus $\int_{\tilde{\delta}}\Omega$ is defined. Again from \eqref{eq:long-exact-sequence} and the residue theorem, we have 
 \begin{align*}
    \int_{\delta}\Omega:=\int_{\tilde{\delta}}\Omega\in \mathbb{C}
 \end{align*} is well-defined. In particular, the complex structure of $D$ is determined by $[\Omega]|_{\operatorname{Im}(\mathrm{H}^1(D,\mathbb{Z}))}$ from the residue formula. 
 The  meromorphic volume form $\Omega$ then determines the period map 
  \begin{align*}
    \phi_{\Omega}:\Lambda(Y,D)\cong \mathrm{H}_2(X,\mathbb{Z})/\operatorname{Im}\mathrm{H}^1(D,\mathbb{Z})  \rightarrow D\cong \operatorname{Pic}^0(D),
  \end{align*} which is similar to the period map of K3 surfaces. Notice that $\phi_{\Omega}$ is independent of the $\mathbb{C}^*$-scaling of $\Omega$. On the other hand, one can have another notion of period $\phi_{(Y,D)}$ in algebraic geometry similar to the one introduced in Gross--Hacking--Keel \cite{GHK},
   \begin{align*}
      \phi_{(Y,D)}\colon  \Lambda(Y,D)&\rightarrow \operatorname{Pic}^0(D) \\  
              L &\mapsto L|_D.
   \end{align*}
   \begin{lem}
   	  The two notions of periods coincide, i.e., $\phi_{\Omega}=\phi_{(Y,D)}$. 
   \end{lem}
\begin{proof}
	 We first consider the case when $Y$ is obtained by blowing up of a smooth cubic 
	 \(\underline{D}\) at distinct points \(x_{1},\ldots,x_{b}\) 
	 Let $Y\rightarrow \mathbf{P}^2$ be the blow-up and $E_i$ the exceptional divisors. 
	 Denote by $D$ the proper transform of \(\underline{D}\) and $H$ the pullback of the hyperplane class on $\mathbf{P}^2$.  	
	 Then $\Lambda(Y,D)$ is spanned by elements of the form $E_i-E_j$ for and 
	 $H-E_i-E_j-E_k$.
	 It is easy to see that both $\phi_{\Omega}$ and $\phi_{(Y,D)}$ are linear and thus it suffices to prove the two coincide on the generators. 
	 Let $p=E_i\cap D, q=E_j\cap D$. Then $\phi_{(Y,D)}(E_j-E_i)=\mathcal{O}_D(q-p)$. On the other hand, one can find a smooth curve $\gamma$ in $D$ connecting $q$ and $p$ and denote by $C_{\gamma}\subseteq X$   the $S^1$-bundle over $\gamma$. One can glue $C_{\gamma}$ into the complement in $E_j-E_i$ of small discs around $p,q$ to obtain a $2$-cycle in $X$ which is homologous to $E_j-E_i$. Denote the $2$-cycle by $C_{ji}$.	 
	 Then one has 
	 \begin{align*}
	    \phi_{\Omega}(E_j-E_i)=\int_{C_{ji}}\Omega=\int_{\gamma}\mathrm{d}z=q-p=\phi_{(Y,D)}(E_j-E_i),
	 \end{align*} where the second equality comes from the residue and the last equality holds via the identification $D\cong \operatorname{Pic}^0(D)$.  Now we consider the case $Y$ is successive blow up (possibly infinitely near) points on $\mathbb{P}^2$ at the smooth cubic. Notice that give a family of pairs  $(Y_t,D_t)$ of successive blow up (possibly infinitely near) points on $\mathbb{P}^2$ at the smooth cubic over a parameter space $T$, the two periods $\phi_{\Omega}$ and $\phi_{(Y,D)}$ are both continuous with respect to $t\in T$. Since one can take $T$ such that generic points correspond to blow up of $\mathbf{P}^2$ at distinct points on a smooth cubic, this proves the lemma for the case when $Y$ is a blow up of $\mathbf{P}^2$.
 When $Y\cong \mathbf{P}^1\times \mathbf{P}^1$, $\Lambda(Y,D)$ is generated by $F_1-F_2$, where $F_i$ are the fibres of different rulings. By choosing the generic fibre representative, we may assume that $F_i$ intersect $D$ at $p_i,q_i$. Choose smooth curves $\gamma_p$ (and $\gamma_q$) connecting $p_1,p_2$ (and $q_1,q_2$ respectively) without intersection. Then the proof is reduced to the above case. 
	 When $Y\cong \mathbb{F}_2$, then $\Lambda(Y,D)$ is generated by the unique $(-2)$-curve and both periods simply vanish. 
\end{proof}	
\begin{proof}[Proof of Theorem \ref{Torelli log CY}] We will first prove the case when $Y$ is not isomorphic $\mathbf{P}^1\times \mathbf{P}^1$ nor $\mathbb{F}_2$. 
 The marking $\mu$ and the period $\phi_{(Y,D)}$ determines a homomorphism $\mathbb{L}_n\rightarrow  \operatorname{Pic}^0(D)$. Corollary 4.4 \cite{M} implies that it uniquely determines the a homomorphism $\mathbb{Z}^{1,n}\rightarrow \operatorname{Pic}(D)$.  Theorem 6.4 \cite{M} says that such a homomorphism recovers the blow up loci of $Y\rightarrow \mathbf{P}^2$ up to the Weyl group action and thus uniquely determines the pair $(Y,D)$ up to isomorphism. 
 
 Now we will consider the case $Y=Y'=\mathbf{P}^1\times \mathbf{P}^1$ and $[F_1],[F_2]$ are the homology classes of two rulings. Assume that $D, D'$ are smooth anti-canonical divisors in $Y=Y'=\mathbf{P}^1\times \mathbf{P}^1$ such that
  \begin{align*}
   \phi_{(Y,D)}([F_1]-[F_2])=\phi_{(Y',D')}([F_1]-[F_2]). 
  \end{align*}
From the group law on elliptic curves $D$ and $D'$,  there exist $p\in Y$, $p'\in Y'$ such that $\phi_{(\tilde{Y},\tilde{D})}=\phi_{(\tilde{Y}',\tilde{D}')}$, where $\tilde{Y}=\mbox{Bl}_{p}Y$, $\tilde{Y}'=\mbox{Bl}_{p'}Y'$ and $\tilde{D}, \tilde{D}'$ are the corresponding proper transforms. Notice that $\tilde{Y}\cong \tilde{Y}'$ are isomorphic to del Pezzo surface of degree $7$. From the previous part of the proof, we have the isomorphism of the pairs $(\tilde{Y},\tilde{D})\cong (\tilde{Y}',\tilde{D}')$. In a del Pezzo surface of degree $7$ there are three $(-1)$-curves and exactly one of them intersects the other two. Therefore, such $(-1)$-curve in $\tilde{Y}$ is identified with a corresponding $(-1)$-curve in $\tilde{Y}'$ via the isomorphism of the pairs $(\tilde{Y},\tilde{D})\cong (\tilde{Y}',\tilde{D}')$. Blowing down such $(-1)$-curves leads to the isomorphism of the pair $(Y,D)\cong (Y',D')$. 
 The proof of the case $\mathbb{F}_2$ is similar. 
\end{proof}

\end{document}